\documentclass[10pt]{amsart}
\usepackage{amsmath,amssymb,latexsym,soul,cite,amsthm,color,enumitem,graphicx,tikz, mathtools,microtype}
\usepackage[colorlinks=true,urlcolor=cadmiumgreen,citecolor=cadmiumgreen,linkcolor=cadmiumgreen,linktocpage,pdfpagelabels,bookmarksnumbered,bookmarksopen]{hyperref}
\definecolor{cadmiumgreen}{rgb}{0.0, 0.42, 0.24}
\usepackage[english]{babel}
\usepackage[left=2.5cm,right=2.5cm,top=2.5cm,bottom=2.5cm]{geometry}

\numberwithin{equation}{section}

\newtheorem{theorem}{Theorem}[section]
\theoremstyle{plain}
\newtheorem{lemma}[theorem]{Lemma}
\theoremstyle{plain}
\newtheorem{proposition}[theorem]{Proposition}
\theoremstyle{plain}
\newtheorem{corollary}[theorem]{Corollary}
\newtheorem{definition}[theorem]{Definition}
\theoremstyle{definition}
\newtheorem{remark}[theorem]{Remark}
\newtheorem{example}[theorem]{Example}

\renewcommand{\l}{\lambda}
\newcommand{\e}{\varepsilon}

\newcommand{\N}{{\mathbb N}}
\newcommand{\Z}{{\mathbb Z}}
\newcommand{\R}{{\mathbb R}}
\newcommand{\eps}{\varepsilon}
\newcommand{\beq}{\begin{equation}}
\newcommand{\eeq}{\end{equation}}
\renewcommand{\le}{\leqslant}
\renewcommand{\ge}{\geqslant}

\newcommand{\restr}[2]{\left.#1\right|_{#2}}

\newenvironment{enumroman}{\begin{enumerate}

}{\end{enumerate}}

\title[Set-valued eigenvalue problems]{Eigenvalue problems for Fredholm operators with set-valued perturbations}

\author[P.\ Benevieri, A.\ Iannizzotto]{Pierluigi Benevieri, Antonio Iannizzotto}

\address[P.\ Benevieri]{Instituto de Matem\'atica e Estat\'istica
\newline\indent
Universidade de S\~ao Paulo
\newline\indent
rua do Mat\~ao 1010, 05508-090 S\~ao Paulo, Brasil}
\email{pluigi@ime.usp.br}

\address[A.\ Iannizzotto]{Department of Mathematics and Computer Science
\newline\indent
University of Cagliari
\newline\indent
Viale L. Merello 92, 09123 Cagliari, Italy}
\email{antonio.iannizzotto@unica.it}


\begin{document}

\begin{abstract}
By means of a suitable degree theory, we prove persistence of eigenvalues and eigenvectors for set-valued perturbations of a Fredholm linear operator. As a consequence, we prove existence of a bifurcation point for a non-linear inclusion problem in abstract Banach spaces. Finally, we provide applications to differential inclusions.
\end{abstract}

\maketitle

\begin{center}
Version of \today\
\end{center}

\section{Introduction}\label{sec1}

\noindent
The present paper is devoted to the study of the following eigenvalue problem with a set-valued perturbation:
\beq\label{pev}
\begin{cases}
Lx-\lambda Cx+\eps\phi(x)\ni 0 \\
x\in\partial\Omega.
\end{cases}
\eeq
Here $L:E\to F$ is a Fredholm linear operator of index 0 between two real Banach spaces $E$ and $F$ s.t. ${\rm ker}\,L\neq 0$, $C$ is another bounded linear operator, $\Omega$ is an open subset of $E$ not necessarily bounded and containing 0, $\phi:\overline\Omega\to 2^F$ is a locally compact, u.s.c.\ set-valued map of $CJ$-type (see Section \ref{sec4} for a precise definition), and $\lambda,\eps\in\R$ are parameters.
\vskip2pt
\noindent
Problem \eqref{pev} can be seen as a set-valued perturbation of a linear eigenvalue problem (which is retrieved for $\eps=0$):
\beq\label{ev}
\begin{cases}
Lx-\lambda Cx= 0 \\
x\in\partial\Omega.
\end{cases}
\eeq
So, it is reasonable to expect that, under suitable assumptions, solutions of \eqref{pev} appear in a neighborhood of the eigenpairs $(x,\lambda)$ of \eqref{ev}. In fact, we show that this is the case for the trivial eigenpairs $(x,0)$, provided ${\rm dim}({\rm ker}\,L)$ is odd, the set $\overline\Omega\cap{\rm ker}\,L$ is compact, and the following transversality condition holds:
\beq\label{trans}
{\rm im}\,L+ C({\rm ker}\,L)=F.
\eeq
More precisely, we denote $\mathcal{S}_0=\partial\Omega\cap{\rm ker}\,L$ the set of trivial solutions of \eqref{ev}. We prove that there exist a rectangle $\mathcal{R}=[-a,a]\times[-b,b]$ ($a,b>0$) and $c>0$ s.t.\ for all $\eps\in[-a,a]$ the set of real parameters $\lambda\in[-b,b]$ for which \eqref{pev} admits a nontrivial solution $x\in E$ with ${\rm dist}(x,\mathcal{S}_0)<c$ is nonempty and depends on $\eps$ by means of an u.s.c.\ set-valued map. Similarly, for all $\eps\in[-a,a]$ the set of vectors $x\in E$ with ${\rm dist}(x,\mathcal{S}_0)<c$ that solve \eqref{pev} for some $\lambda\in[-b,b]$ is nonempty and depends on $\eps$ by means of an u.s.c.\ set-valued map. Using such persistence results, we prove that $\mathcal{S}_0$ contains at least one bifurcation point, i.e., a trivial solution $x_0$ s.t.\ any neighborhood of $x_0$ in $E$ contains a nontrivial solution. 
\vskip2pt
\noindent
Our results are an extension of those of \cite{BCFP}, where the first author, with A.\ Calamai, M.\ Furi, and M.P.\ Pera, considered a \eqref{ev}-type eigenvalue problem perturbed by a single-valued nonlinear map. The origin of the study of this type of nonlinear eigenvalue problems goes back to a work of R.\ Chiappinelli \cite{C} in which the author investigates a {\em persistence} property of the eigenvalues and eigenvectors of the system
\beq\label{Chiappinelli}
\begin{cases}
Lx + \e N(x)= \l x \\
\|x\|=1,
\end{cases}
\eeq
where $L$ is a self-adjoint operator defined on a real Hilbert space $H$, $N \colon H \to H$ is a nonlinear continuous (single-valued) map, $\e$, $\l$ still are real parameters. Under the assumptions that $\l_0 \in \R$ is an isolated {\em simple} eigenvalue of $L$ and that $N$ is Lipschitz continuous, Chiappinelli proves that there exist two $H$-valued Lipschitz curves, $\e \mapsto x\sp1_\e$ and $\e \mapsto x\sp2_\e$, defined in a neighborhood $V$ of $0$ in $\R$, as well as two real Lipschitz functions, $\e \mapsto \l\sp1_\e$ and $\e \mapsto \l\sp2_\e$, s.t.\ for $i=1,2$ and $\e \in V$ one has
\[Lx\sp{i}_\e + \e N(x\sp{i}_\e) = \l\sp{i}_\e x\sp{i}_\e, \quad \|x\sp{i}_\e\|=1,\]
i.e., the triples $(x^i_\eps,\eps,\lambda^i_\eps)$ solve \eqref{Chiappinelli} for all $\eps\in V$. In particular, when $\e = 0$, these four functions satisfy $x\sp{i}_0=x\sp{i}$, $\l\sp{i}_0=\l_0$, where $x\sp1$ and $x\sp2$ are the two unit eigenvectors of $L$ corresponding to the simple eigenvalue~$\l_0$. After the result of Chiappinelli, in a series of papers \cite{CFP,CFP1,CFP2,CFP3} the above property of local persistence of the eigenvalues and eigenvectors was extended to the case in which the multiplicity of the eigenvalue $\l_0$ is bigger than one. In particular, in \cite{BCFP} persistence results are obtained in the more general framework of real Banach spaces.
\vskip2pt
\noindent
We proceed here in the general spirit of \cite{BCFP} in which the authors use a topological approach based on a concept of degree, developed in \cite{BCF1, BCF2}, for a class of noncompact (single-valued) perturbations of Fredholm maps of index zero between Banach spaces. On the other hand, introducing in our work a set-valued perturbation requires a more general degree theory for set-valued maps, which extends Brouwer's degree for nonlinear maps on $C^1$-manifolds. Such a degree theory has been introduced in \cite{OZZ} and redefined in \cite{BZ} by a precise notion of orientation for set-valued perturbations of nonlinear Fredholm maps between Banach spaces. The concept of orientation used in \cite{BZ} (and reproduced here) is a natural extension of a notion of  orientation for nonlinear Fredholm maps in  Banach spaces presented in \cite{BF,BF1} and on which is also based the approach in \cite{BCFP}. This orientation actually simplifies the method followed to define the degree in \cite{OZZ}, based on the so called concept of oriented Fredholm structure, introduced by Elworty and Tromba in \cite{ET,ET1} (where an orientation is constructed on the source and targets Banach spaces and manifolds).
\vskip2pt
\noindent
In order to help the reader, most of our paper (Sections \ref{sec2}-\ref{sec5}) is devoted to the construction of the orientation and degree for the set-valued perturbations of  Fredholm maps. Then, in Section \ref{sec6} we prove our persistence and bifurcation results. Finally, in Section \ref{sec7} we will provide some examples and applications of our abstract theorems, showing that our assumptions are satisfied in quite natural situations and may lead to new existence results for differential inclusions. Precisely, we will consider the following ordinary differential inclusion with Neumann boundary conditions and an integral constraint:
\beq\label{odi}
\begin{cases}
u''+u'-\lambda u+\eps\Phi(u)\ni 0 \ \text{in $[0,1]$} \\
u'(0)=u'(1)=0 \\
\|u\|_1=1.
\end{cases}
\eeq
Here $\Phi(u):[0,1]\to 2^\R$ is a set-valued map depending on $u$, to be chosen according to several requirements (three different examples will be presented). We shall prove that the transversality condition \eqref{trans} holds, and hence problem \eqref{odi} admits at least one bifurcation point.
\vskip2pt
\noindent
{\bf Notation:} Whenever $E$, $F$ are Banach spaces, we denote by $\mathcal{L}(E,F)$ the space of bounded linear operators from $E$ into $F$ (in particular, $\mathcal{L}(E)=\mathcal{L}(E,E)$). We shall use the term {\em operator} for linear functions, and {\em map} for nonlinear ones.

\section{A remark on orientation and transversality}\label{sec2}

\noindent
In this preliminary section we recall some facts regarding the classical notions of orientation and transversality in finite dimension. We assume that the reader is familiar with the notion of orientation for finite-dimensional Banach manifolds and spaces. Let $M$ be a real $C^1$-manifold, $F$ be a real vector space s.t.
\[{\rm dim}(M)={\rm dim}(F)<\infty.\]

\begin{definition}\label{transverse}
A subspace $F_1\subseteq F$ and a map $g\in C^1(M,F)$ are \emph{transverse} if for all $x\in M$
\[{\rm im}\,Dg(x)+F_1=F.\]
\end{definition}

\noindent
The map $g$ is backward orientation-preserving:

\begin{lemma}\label{backorient}
Let $M$, $F$ be oriented, and $F_1\subseteq F$, $g\in C^1(M,F)$ be transverse. Then, any orientation of $F_1$ induces an orientation of $M_1=g^{-1}(F_1)$.
\end{lemma}
\begin{proof}
By regularity of $g$, $M_1$ is a $C^1$-submanifold of $M$ with
\[{\rm dim}(M_1)={\rm dim}(F_1).\]
Fix $x\in M_1$, and let $T_x(M)$, $T_x(M_1)$ be the tangent spaces to $M$, $M_1$, respectively, at $x$. Then we have
\[T_x(M_1)=(Dg(x))^{-1}(F_1).\]
Let $E_0$ be a direct complement to $T_x(M_1)$ in $T_x(M)$, and $F_0=Dg(x)(E_0)$. Then the restriction $\restr{Dg(x)}{E_0}\in\mathcal{L}(E_0,F_0)$ is an isomorphism and $F_0\oplus F_1=F$. Now let $F_1$ be oriented so that any two positively oriented bases of $F_0$, $F_1$ (in this order) form a positively oriented basis of $F$. Thus, we can orient $E_0$ so that $\restr{Dg(x)}{E_0}$ is orientation-preserving.
\vskip2pt
\noindent
Similarly, we can orient $T_x(M_1)$ so that any two positively oriented bases of $E_0$, $T_x(M_1)$ (in this order) form a positively oriented basis of $T_x(M)$. Then, this pointwise choice induces a global orientation on $M_1$ (see \cite[p.\ 100]{GP} for details).
\end{proof}

\noindent
By Lemma \ref{backorient} we have a natural way to orient $M_1$:

\begin{definition}\label{orpreim}
Let $M$, $F$, $F_1\subseteq F$ be oriented, and $g\in C^1(M,F)$ be transverse to $F_1$. The manifold $M_1=g^{-1}(F_1)$, with the orientation induced by that of $F_1$ is an \emph{oriented $g$-preimage} of $F_1$.
\end{definition}

\noindent
Now let $f\in C(M,F)$, and choose $y\in F$ s.t.\ $f^{-1}(y)\subset M$ is compact. Brouwer's degree for the triple $(f,M,y)$ is defined and denoted by
\[{\rm deg}_B(f,M,y)\in\Z.\]
For the definition and properties of Brouwer's degree (both on open sets and manifolds) we refer to \cite{L,N}. We only need to add the following reduction property:

\begin{proposition}\label{reduction}
Let $M$, $F$, $F_1\subset F$ be oriented, $g\in C^1(M,F)$ be transverse to $F_1$, $M_1$ be the oriented $g$-preimage of $F_1$. Moreover, let $f\in C(M,F)$, $y\in F_1$ be s.t.\ $f^{-1}(y)$ is compact and
\[(f-g)(M)\subseteq F_1.\]
Finally, let $f_1=\restr{f}{M_1}$. Then,
\[{\rm deg}_B(f,M,y)={\rm deg}_B(f_1,M_1,y).\]
\end{proposition}
\begin{proof}
First we note that for all $x\in M_1$
\[f(x)=g(x)+(f-g)(x)\in F_1,\]
so $f_1\in C(M_1,F_1)$. In particular, $f^{-1}_1(y)=f^{-1}(y)$ is a compact subset of $M_1$. We orient $M_1$ and $F_1$ as in Lemma \ref{backorient}, so we can define Brouwer's degree for the triple $(f_1,M_1,y)$. Now, the conclusion follows from \cite[Lemma 4.2.3]{L}.
\end{proof}

\section{Orientation for Fredholm maps}\label{sec3}

\noindent
In order to develop a degree theory, we need a precise notion of {\em orientability} for Fredholm operators and maps. The one we are going to recall here was introduced in \cite{BF,BF1}.
\vskip2pt
\noindent
Let $E$, $F$ be two (possibly, infinite-dimensional) real Banach spaces. We first recall a basic definition:

\begin{definition}\label{fredop}
A bounded linear operator $L\in \mathcal{L}(E,F)$ is a \emph{Fredholm operator} of index $k\in\Z$, if
\begin{enumroman}
\item\label{fredop1} ${\rm dim}({\rm ker}\,L),\,{\rm dim}({\rm coker}\,L)<\infty$;
\item\label{fredop2} ${\rm dim}({\rm ker}\,L)-{\rm dim}({\rm coker}\,L)=k$.
\end{enumroman}
The set of such operators is denoted $\Phi_k(E,F)$.
\end{definition}

\noindent
It is known that $\Phi_k(E,F)\subset\mathcal{L}(E,F)$ is open for all $k\in\Z$. We are mainly interested in $\Phi_0(E,F)$, the set of Fredholm operators of index $0$, also denoted $\Phi_0$-operators. The following construction leads to a notion of orientation for such operators:

\begin{definition}\label{corrector}
Let $L\in\Phi_0(E,F)$, $A\in\mathcal{L}(E,F)$. $A$ is a \emph{corrector} of $L$, if
\begin{enumroman}
\item\label{corrector1} ${\rm dim}({\rm im}\,A)<\infty$ (finite rank);
\item\label{corrector2} $L+A\in\mathcal{L}(E,F)$ is an isomorphism.
\end{enumroman}
The set of correctors of $L$ is denoted $\mathcal{C}(L)$.
\end{definition}

\noindent
Clearly, $\mathcal{C}(L)\neq\emptyset$ for all $L\in\Phi_0(E,F)$. Following \cite{BF}, we define an equivalence relation in $\mathcal{C}(L)$. Let $A,B\in\mathcal{C}(L)$, and set
\[T=(L+B)^{-1}(L+A), \quad  \quad K=I-T=(L+B)^{-1}(B-A).\]
By Definition \ref{corrector}, $T\in\mathcal{L}(E)$ is an automoprphism, and $K\in\mathcal{L}(E)$ has finite rank. Let $E_0\subseteq E$ be a non-trivial finite-dimensional subspace s.t.\ ${\rm im}\,K\subseteq E_0$, and set $T_0=\restr{T}{E_0}$. We note that $T_0\in\mathcal{L}(E_0)$ and is an automorphism as well. Indeed, $T_0$ is injective by injectivity of $T$, and for all $x\in E_0$ we have 
\[T_0(x)=x-K(x)\in E_0,\]
so $T_0$ is surjective as well (recall that ${\rm dim}(E_0)<\infty$). Thus, as soon as we fix a basis for $E_0$, the determinant of $T_0$ is well defined and denoted ${\rm det}\,T_0\in\R\setminus\{0\}$. A remarkable fact is that ${\rm det}\,T_0$ does not depend on the choice of $E_0$ (by choosing the same basis in $E_0$ both as the domain and as the codomain of $T_0$), so we can provide $T$ with a uniquely defined determinant by setting
\[{\rm det}\,T={\rm det}\,T_0.\]
The above notion of determinant for linear operators between (possibly) infinite dimensional spaces can be found in \cite{K}.

\begin{definition}\label{L-equiv}
Let $L\in\Phi_0(E,F)$. Two correctors $A,B\in\mathcal{C}(L)$ are \emph{$L$-equivalent}, if
\[{\rm det}\,\big((L+B)^{-1}(L+A)\big)>0.\]
\end{definition}

\noindent
It is easily seen that $L$-equivalence is actually an equivalence relation, splitting $\mathcal{C}(L)$ into two equivalence classes. Now we can define a notion of orientation for $\Phi_0$-operators:

\begin{definition}\label{fredorient}
Let $L\in\Phi_0(E,F)$:
\begin{enumroman}
\item\label{fredorient1} an \emph{orientation} of $L$ is any $L$-equivalence class $\alpha\subset\mathcal{C}(L)$, then the pair $(L,\alpha)$ is an \emph{oriented} $\Phi_0$-operator, moreover a corrector $A\in\mathcal{C}(L)$ is \emph{positive} for $(L,\alpha)$ if $A\in\alpha$, \emph{negative} if $A\in\mathcal{C}(L)\setminus\alpha$;
\item\label{fredorient2} if $L$ is an isomorphism, then $\alpha\subset\mathcal{C}(L)$ is the \emph{natural orientation} of $L$, if $0\in\alpha$, and in such case $(L,\alpha)$ is \emph{naturally oriented};
\item\label{fredorient3} if $(L,\alpha)$ is an oriented $\Phi_0$-operator, its \emph{sign} is defined as follows:
\[{\rm sign}(L,\alpha)=\begin{cases}
+1 & \text{if $(L,\alpha)$ is a naturally oriented isomorphism} \\
-1 & \text{if $(L,\alpha)$ is a non-naturally oriented isomorphism} \\
0 & \text{if $(L,\alpha)$ is not an isomorphism.} \\
\end{cases}\]
\end{enumroman}
\end{definition}

\noindent
Let $(L,\alpha)$ be an oriented $\Phi_0$-operator, $A\in\alpha$ be a positive corrector. Since the set of isomorphisms is open in $\mathcal{L}(E,F)$, we can find a neighborhood $\mathcal{U}\subset\Phi_0(E,F)$ of $L$ s.t.\ $A\in\mathcal{C}(T)$ for all $T\in\mathcal{U}$. So, any operator $T\in\mathcal{U}$ can be oriented so that $A\in\mathcal{C}(T)$ is a positive corrector. In such a way, any orientation of $L$ induces orientations of nearby $\Phi_0$-operators, which allows us to define orientability of $\Phi_0(E,F)$-valued maps:

\begin{definition}\label{orientcont}
Let $X$ be a topological space, $h\in C(X,\Phi_0(E,F))$. An \emph{orientation} of $h$ is a map $\alpha$ defined in $X$, s.t.\ for all $x\in X$
\begin{enumroman}
\item\label{orientcont1} $\alpha(x)$ is an orientation of $h(x)\in\Phi_0(E,F)$;
\item\label{orientcont2} there exist $A\in\alpha(x)$ and a neighborhood $V\subset X$ of $x$, s.t.\ $A\in\alpha(y)$ for all $y\in V$ (continuity).
\end{enumroman}
The map $h$ is \emph{orientable} if it admits an orientation, and in such case $(h,\alpha)$ is an \emph{oriented} $\Phi_0(E,F)$-valued map.
\end{definition}

\noindent
Now we can consider (nonlinear) Fredholm maps:

\begin{definition}\label{fredmap}
Let $\Omega\subseteq E$ be open. A map $g\in C^1(\Omega,F)$ is a $\Phi_0$-\emph{map}, if $Dg(x)\in\Phi_0(E,F)$ for all $x\in\Omega$.
\end{definition}

\noindent
For instance, any Fredholm operator $L\in\Phi_0(E,F)$ is a $\Phi_0$-map, since $DL(x)=L$ for all $x\in E$.

\begin{definition}\label{fredmapor}
Let $\Omega\subseteq E$ be open, $g\in C^1(\Omega,F)$ be a $\Phi_0$-map:
\begin{enumroman}
\item\label{fredmapor1} an \emph{orientation} of $g$ is any orientation of $Dg\in C(\Omega,\Phi_0(E,F))$ (Definition \ref{orientcont});
\item\label{fredmapor2} the map $g$ is \emph{orientable} if it admits an orientation $\alpha$, and in such case $(g,\alpha)$ is an \emph{oriented} $\Phi_0$-map.
\end{enumroman}
\end{definition}

\noindent
The existence (and number) of orientations of a $\Phi_0$-map depend mainly on the topology of its domain (see \cite{BF} for the proof):

\begin{proposition}\label{fredmaptop}
Let $\Omega\subseteq E$ be open, $g\in C^1(\Omega,F)$ be a $\Phi_0$-map:
\begin{enumroman}
\item\label{fredmaptop1} if $g$ is orientable, then it admits at least two orientations;
\item\label{fredmaptop2} if $g$ is orientable and $\Omega$ is connected, then $g$ admits exactly two orientations;
\item\label{fredmaptop3} if $\Omega$ is simply connected, then $g$ is orientable.
\end{enumroman}
\end{proposition}

\noindent
Another important use of Definition \ref{orientcont} is towards orientation of Fredholm homotopies:

\begin{definition}\label{fredhom}
Let $\Omega\subseteq E$ be open. A map $h\in C(\Omega\times[0,1],F)$ is a \emph{$\Phi_0$-homotopy}, if 
\begin{enumroman}
\item\label{fredhom1} $h(\cdot,t)$ is a $\Phi_0$-map for all $t\in[0,1]$;
\item\label{fredhom2} the map $(x,t)\mapsto D_xh(x,t)$ is continuous from $\Omega\times[0,1]$ into $\Phi_0(E,F)$, where we denote by $D_xh(x,t)$ the derivative of $h(\cdot,t)$ at $x$.
\end{enumroman}
\end{definition}

\noindent
Note that no differentiability in $t$ is required. Condition \ref{fredhom2} here is crucial, as it allows us to apply Definition \ref{orientcont} to the map $(x,t)\mapsto D_xh(x,t)$, and thus define a notion of orientation for $\Phi_0$-homotopies:

\begin{definition}\label{fredhomor}
Let $\Omega\subseteq E$ be open, $h\in C(\Omega\times[0,1],F)$ be a $\Phi_0$-homotopy:
\begin{enumroman}
\item\label{fredhomor1} an \emph{orientation} of $h$ is any orientation of $D_xh\in C(\Omega\times[0,1],\Phi_0(E,F))$ (Definition \ref{orientcont});
\item\label{fredhomor2} the homotopy $h$ is \emph{orientable} if it admits an orientation $\alpha$, and in such case $(h,\alpha)$ is an \emph{oriented} $\Phi_0$-homotopy.
\end{enumroman}
\end{definition}

\noindent
Let $(h,\alpha)$ be an oriented $\Phi_0$-homotopy. Clearly, $\alpha$ induces an orientation $\alpha_t$ of the $\Phi_0$-map $h(\cdot,t)$, for all $t\in[0,1]$. Remarkably, the converse is also true, as shown by the following result on continuous transportation of orientations (see \cite[Theorem 3.14]{BF}):

\begin{proposition}\label{transport}
Let $\Omega\subseteq E$ be open, $h\in C(\Omega\times[0,1],F)$ be a $\Phi_0$-homotopy, $t\in[0,1]$ be s.t.\ $h(\cdot,t)\in C^1(\Omega,F)$ admits an orientation $\alpha_t$. Then, there exists a unique orientation $\alpha$ of $h$ which induces $\alpha_t$.
\end{proposition}

\noindent
We conclude this section by establishing a link between the orientation of Fredholm maps and that of manifolds:

\begin{proposition}\label{fredpreimage}
Let $\Omega\subseteq E$ be open, $g\in C^1(\Omega,F)$ be an orientable $\Phi_0$-map, $F_1\subseteq F$ be a finite-dimensional subspace, transverse to $g$, and $M_1=g^{-1}(F_1)$. Then:
\begin{enumroman}
\item\label{fredpreimage1} $M_1\subseteq E$ is a $C^1$-manifold with ${\rm dim}(M_1)={\rm dim}(F_1)$;
\item\label{fredpreimage2} $M_1$ is orientable;
\item\label{fredpreimage3} any orientation of $g$ and any orientation of $F_1$ induce an orientation of $M_1$.
\end{enumroman}
\end{proposition}
\begin{proof}
Assertion \ref{fredpreimage1} is obvious (see Section \ref{sec2}). Assertion \ref{fredpreimage2} follows from \cite[Remark 2.5, Lemma 3.1]{BF}.
\vskip2pt
\noindent
We prove \ref{fredpreimage3}. Let $\alpha$ be an orientation of $g$, and $x\in M_1$. By Definition \ref{fredmapor}, $\alpha(x)$ is an orientation of $Dg(x)\in\Phi_0(E,F)$. By transversality (Definition \ref{transverse}), we can find $A\in\alpha(x)$ s.t.\ ${\rm im}\,A\subseteq F_1$. Indeed, since $Dg(x)\in\Phi_0(E,F)$, we can split both Banach spaces as follows:
\[E=Dg(x)^{-1}(F_{1})\oplus E_2, \ F=F_1\oplus F_{2},\]
where $E_{2}$ is any direct complement of $Dg(x)^{-1}(F_{1})$ and $F_{2}:=Dg(x)\left(E_{2}\right)$. Observe that ${\rm ker}\,Dg(x)\subseteq Dg(x)^{-1}(F_{1})$ and the latter has the same dimension as $F_{1}$.  So we rephrase $Dg(x)$ as
\[Dg(x)=\begin{bmatrix}
L_{1,1} & 0 \\
0 & L_{2,2}
\end{bmatrix},\]
where $L_{2,2}\in\mathcal{L}(E_2,F_{2})$ is an isomorphism. We may choose $A\in\mathcal{L}(E,F)$ with the structure
\begin{equation}
\label{corrbm}A=\begin{bmatrix}
A_{1,1} & 0 \\
0 & 0
\end{bmatrix},
\end{equation}
where $A_{1,1}+L_{1,1}\in\mathcal{L}(L^{-1}(F_{1}),F_1)$ is an isomporhism. So $A\in\mathcal{C}(Dg(x))$ and ${\rm im}\,A\subseteq F_1$. Choosing $A_{1,1}$ in such a way that $A\in \alpha(x)$ and assigning an orientation to $F_{1}$, we orient the tangent space $T_x(M_1)\subset E$ so that the isomorphism
\[(Dg(x)+A)_{|T_x(M_1)}\in\mathcal{L}(T_x(M_1),F_1)\]
is orientation-preserving. As proved in \cite{BF2}, such orientation of $T_x(M_1)$ does not depend on $A$. This pointwise choice induces a global orientation on $M_1$.
\end{proof}

\noindent
We can now give a Fredholm analogue of Definition \ref{orpreim}:

\begin{definition}\label{fredorpreim}
Let $\Omega\subseteq E$ be open, $(g,\alpha)$ be an oriented $\Phi_0$-map, $F_1\subseteq F$ be a finite-dimensional subspace, transverse to $g$, and $M_1=g^{-1}(F_1)$. With the orientation induced by $\alpha$ and the orientation of $F_1$, $M_1$ is an \emph{oriented $(\Phi_0,g)$-preimage} of $F_1$.
\end{definition}

\begin{remark}
In what follows, we will denote an oriented $\Phi_0$-operator $(L,\alpha)$ simply by $L$, as long as no confusion arises. We will do the same for oriented $\Phi_0$-maps, $\Phi_0$-homotopies, and so on.
\end{remark}

\section{Topological properties of set-valued maps}\label{sec4}

\noindent
In this section we point out some definitions and properties of set-valued maps between metric spaces, referring the reader to \cite{G} for details. Let $X$, $Y$ be metric spaces with distance functions $d_X$, $d_Y$, respectively. Then $X\times Y$ is a metric space under the distance
\[d\big((x,y),(x',y')\big)=\max\big\{d_X(x,x'),\,d_Y(y,y')\big\}.\]
For all $A\subset X$, $x\in X$ we set
\[{\rm dist}(x,A)=\inf_{z\in A}\,d_X(x,z),\]
and for all $\eps>0$ we set
\[B_\eps(A)=\big\{x\in X:\,{\rm dist}(x,A)<\eps\big\}\]
(if $A=\{x\}$, then we set $B_\eps(A)=B_\eps(x)$). A set-valued map $\phi:X\to 2^Y$ is a map from $X$ to the set of all parts of $Y$. {\em We will always assume that $\phi$ is compact-valued, i.e., that $\phi(x)\subseteq Y$ is either $\emptyset$ or compact, for all $x\in X$.} The graph of $\phi$ is defined by
\[{\rm graph}\,\phi=\big\{(x,y)\in X\times Y:\,y\in\phi(x)\big\}.\]
We also recall a classical definition:

\begin{definition}\label{usc}
A set-valued map $\phi:X\to 2^Y$ is \emph{upper semi-continuous (u.s.c.)}, if for all open $V\subseteq Y$ the set
\[\phi^+(V)=\big\{x\in X:\,\phi(x)\subseteq V\big\}\]
is open.
\end{definition}

\noindent
A remarkable property of u.s.c.\ set-valued maps, is that they preserve compactness. Any (single-valued) map $f:X\to Y$ coincides with the set-valued map $\phi(x)=\{f(x)\}$, in such case $\phi$ is u.s.c.\ iff $f$ is continuous. Another interesting special case is that of set-valued maps with a compact graph (see also \cite[Remark 2.1]{BCFP}):

\begin{lemma}\label{compgr}
Let $\mathcal{K}\subset X\times Y$ be compact, and set for all $x\in X$
\[\phi(x)=\big\{y\in Y:\,(x,y)\in\mathcal{K}\big\}.\]
Then, $\phi:X\to 2^Y$ is u.s.c.\
\end{lemma}
\begin{proof}
We argue by contradiction. Let $V\subset Y$ be open, s.t.\ $\phi^+(V)\subset X$ is not open. Then, we can find a sequence $(x_n)$ in $X\setminus\phi^+(V)$ s.t.\ $x_n\to x$ for some $x\in\phi^+(V)$. That means that for all $n\in\N$ $\phi(x_n)$ is not empty since is not contained in $V$. Hence, there exists $y_n\in\phi(x_n)\setminus V$, while $\phi(x)\subset V$. The sequence $(x_n,y_n)$, which lies in $\mathcal{K}$, admits a subsequence, still denoted $(x_n,y_n)$, converging to $(x,y)\in\mathcal{K}$. So $y_n\to y$, and since $Y\setminus V$ is closed we have $y\in Y\setminus V$. So $y\notin\phi(x)$, against $(x,y)\in\mathcal{K}$.
\end{proof}

\noindent
We introduce the notion of {\em approximability}:

\begin{definition}\label{approx}
Let $\phi:X\to 2^Y$:
\begin{enumroman}
\item\label{approx1} for all $\eps>0$, $f\in C(X,Y)$ is an \emph{$\eps$-approximation} of $\phi$, if for all $x\in X$ there exists $x'\in B_\eps(x)$ s.t.\ $f(x)\in B_\eps(\phi(x'))$ (the set of $\eps$-approximations of $\phi$ is denoted $B_\eps(\phi)$);
\item\label{approx2} $\phi$ is \emph{approximable}, if $B_\eps(\phi)\neq\emptyset$ for all $\eps>0$.
\end{enumroman}
\end{definition}

\noindent
Note that all approximations of a set-valued map are required to be {\em continuous}. A characterization (whose proof is an obvious consequence of Definition \ref{approx}):

\begin{lemma}\label{apprchar}
Let $\phi:X\to 2^Y$, $\eps>0$, $f\in C(X,Y)$. Then, the following are equivalent:
\begin{enumroman}
\item\label{apprchar1} $f\in B_\eps(\phi)$;
\item\label{apprchar2} $f(x)\in B_\eps\big(\phi(B_\eps(x))\big)$ for all $x\in X$;
\item\label{apprchar3} ${\rm graph}\,f\subseteq B_\eps({\rm graph}\,\phi)$.
\end{enumroman}
\end{lemma} 

\noindent
Approximation of an u.s.c.\ set-valued map is a special case, enjoying several properties (see \cite[Proposition 22.3]{G}):

\begin{proposition}\label{apprusc}
Let $\phi:X\to 2^Y$ be u.s.c. Then:
\begin{enumroman}
\item\label{apprusc1} for all compact $X_1\subseteq X$, $\eps>0$ there exists $\delta>0$ s.t.\ for all $f\in B_\delta(\phi)$ we have $\restr{f}{X_1}\in B_\eps(\restr{\phi}{X_1})$;
\item\label{apprusc2} if $X$ is compact, then for any metric space $Z$, $g\in C(Y,Z)$, and $\eps>0$ there exists $\delta>0$ s.t.\ for all $f\in B_\delta(\phi)$ we have $g\circ f\in B_\eps(g\circ\phi)$;
\item\label{apprusc3} if $X$ is compact, then for any u.s.c.\ set-valued map $\psi:X\times[0,1]\to 2^Y$, $\eps>0$, and $t\in[0,1]$ there exists $\delta>0$ s.t.\ for all $f\in B_\delta(\psi)$ we have $f(\cdot,t)\in B_\eps(\psi(\cdot,t))$;
\item\label{apprusc4} for any metric space $Z$, any u.s.c.\ set-valued map $\psi:X\to 2^Z$, and $\eps>0$ there exists $\delta>0$ s.t.\ for all $f\in B_\delta(\phi)$, $g\in B_\delta(\psi)$ we have $(f,g)\in B_\eps(\phi\times\psi)$.
\end{enumroman}
\end{proposition}

\noindent
Approximability of a set-valued map is strongly influenced by the topology of its values, the easiest case being in general that of convex-valued maps between Banach spaces. In the general case of a metric space, convexity makes no sense and it must be replaced by a more general notion. We recall from \cite{G} some definitions and properties (here $\mathbb{S}^{n-1}$, $\mathbb{B}^n$ denote the unit sphere and closed  ball, respectively, in $\R^n$):

\begin{definition}\label{aspheric}
A set $A\subset Y$ is \emph{aspheric}, if for any $\eps>0$ there exists $\delta\in(0,\eps)$ s.t.\ for all $n\in\N$ and all $g\in C(\mathbb{S}^{n-1},B_\delta(A))$ there is $\tilde g\in C(\mathbb{B}^n,B_\eps(A))$ s.t.\ $\restr{\tilde g}{\mathbb{S}^{n-1}}=g$.
\end{definition}

\noindent
The following characterization of aspheric sets holds in $ANR$-spaces ({\em absolute neighborhood retracts}, see \cite[Definition 1.7]{G}):

\begin{proposition}\label{asphchar}
Let $Y$ be an $ANR$-space, $A\subseteq Y$. Then, the following are equivalent:
\begin{enumroman}
\item\label{asphchar1} $A$ is aspheric;
\item\label{asphchar2} there exists a decreasing sequence $(A_n)$ of compact, contractible subsets of $Y$ s.t.\ $\cap_{n=1}^\infty A_n=A$ ($R_\delta$-set).
\end{enumroman}
\end{proposition}

\noindent
We go back to set-valued maps:

\begin{definition}\label{jmap}
A set-valued map $\phi:X\to 2^Y$ is a $J$-\emph{map}, if $\phi$ is u.s.c.\ and $\phi(x)$ is aspheric for all $x\in X$. The set of $J$-maps from $X$ to $Y$ is denoted by $J(X,Y)$.
\end{definition}

\noindent
Some sufficient conditions:

\begin{lemma}\label{jmapsuff}
Let $Y$ be an $ANR$-space, $\phi:X\to 2^Y$ be u.s.c., and one of the following hold:
\begin{enumroman}
\item\label{jmapsuff1} $\phi(x)$ an $R_\delta$-set for all $x\in X$;
\item\label{jmapsuff2} $\phi(x)$ is an $AR$-set for all $x\in X$.
\end{enumroman}
Then, $\phi\in J(X,Y)$.
\end{lemma}

\begin{remark}
\label{convexity}
In particular, if either $\phi$ has contractible values \ref{jmapsuff1}, or $Y$ is a Banach space and $\phi$ has convex values \ref{jmapsuff2}, then $\phi\in J(X,Y)$.
\end{remark}

\noindent
The purpose of this topological digression is to introduce a class of approximable set-valued maps (see \cite[Theorems 23.8, 23.9]{G}):

\begin{proposition}\label{jmapappr}
Let $X$ be a compact $ANR$-space, $\phi\in J(X,Y)$. Then:
\begin{enumroman}
\item\label{jmapappr1} $\phi$ is approximable;
\item\label{jmapappr2} for all $\eps>0$ there exists $\delta_\eps>0$ s.t.\ for all $\delta\in(0,\delta_\eps)$, $f,g\in B_\delta(\phi)$ we can find a homotopy $h\in C(X\times[0,1],Y)$ s.t.\ $h(\cdot,0)=f$, $h(\cdot,1)=g$, and $h(\cdot,t)\in B_\eps(\phi)$ for all $t\in[0,1]$.
\end{enumroman}
\end{proposition}

\noindent
Assertion \ref{jmapappr2} can be described as homotopy-stability of approximations. In our results, we shall need a slightly more general class of set-valued maps:

\begin{definition}\label{cjmap}
A set-valued map $\phi:X\to 2^Y$ is a $CJ$-\emph{map}, if there exist a metric space $Z$, $\psi\in J(X,Z)$, and $k\in C(Z,Y)$ s.t.\ $\phi=k\circ\psi$. The set of $CJ$-maps from $X$ to $Y$ is denoted by $CJ(X,Y)$.
\end{definition}

\noindent
From Propositions \ref{apprusc} and \ref{jmapappr} we clearly have:

\begin{proposition}\label{cjmapappr}
Let $X$ be a compact $ANR$-space, $\phi\in CJ(X,Y)$. Then, \ref{jmapappr1} and \ref{jmapappr2} of Proposition \ref{jmapappr} hold.
\end{proposition}
\begin{proof}
We prove \ref{jmapappr1}, the argument for \ref{jmapappr2} being analogous. Let $\phi=k\circ\psi$ be as in Definition \ref{cjmap}, and fix $\eps>0$. Since $X$ is compact, so is $\psi(X)\subset Z$. Hence, by Cantor-Heine's theorem we can find $\delta\in(0,\eps]$ s.t.\ for all $z',z''\in\psi(X)$ we have
\[d_Z(z',z'')<\delta \ \Longrightarrow \ d_Y(k(z'),k(z''))<\eps.\]
By Proposition \ref{jmapappr}, there exists $p\in B_\delta(\psi)$. Set $f=k\circ p\in C(X,Y)$. Then, for all $x\in X$ there exists $x'\in B_\delta(x)$ s.t.\ $p(x)\in B_\delta(\psi(x'))$, i.e., there is $z\in\psi(x')$ satisfying $d_Z(p(x),z)<\delta$. By the implication above we have
\[{\rm dist}(f(x),\phi(x'))\le d_Y(k(p(x)),k(z))<\eps,\]
so $f\in B_\eps(\phi)$.
\end{proof}

\section{Degree for multitriples}\label{sec5}

\noindent
In this section we develop a degree theory for set-valued maps, extending Brouwer's degree. This degree has been presented in \cite{BZ} and its construction  basically follows \cite{OZZ}, except for the notion of orientation.  In fact, our approach is based on the notion of orientation for Fredholm maps, introduced in \cite{BF1,BF2} and recalled here in Section \ref{sec3}, while the construction in \cite{OZZ} makes use of the concept of oriented Fredholm structures, introduced in \cite{ET,ET1}. For a comprehensive presentation of degree theory for set-valued maps the reader can see the very rich textbook of M.\ V\"ath \cite{Va}. Throughout this section $E$, $F$ are real Banach spaces and $\Omega\subseteq E$ is an open set.

\begin{definition}\label{triple}
Let $g\in C^1(\Omega,F)$ be an oriented $\Phi_0$-map, $U\subseteq\Omega$ be open, and $\phi\in CJ(\Omega,F)$ be locally compact. $(g,U,\phi)$ is an \emph{admissible (multi)-triple}, if the coincidence set
\[C(g,U,\phi)=\big\{x\in U:\,g(x)\in\phi(x)\big\}\]
is compact.
\end{definition}

\noindent
We construct our degree as an integer-valued function defined on the set of admissible triples. First we assume
\beq\label{fin}
{\rm dim}(\phi(U))<\infty.
\eeq
Since $C(g,U,\phi)$ is compact, we can find an open neighborhood $W\subset U$ of $C(g,U,\phi)$ and a subspace $F_1\subseteq F$ s.t.\ ${\rm dim}(F_1)=m<\infty$, $\phi(U)\subset F_1$ (by virtue of \eqref{fin}), and $F_1$ is transverse to $g$ in $W$ (Definition \ref{transverse}), as it can be seen as follows: given any $x\in C(g,U,\phi)$, take a finite-dimensional subspace $F_{x}$ of $F$ containing $\phi(U)$ and transverse to $g$ at $x$. This is possible since $Dg(x)$ is Fredholm. By the continuity of $z\mapsto Dg(z)$, there exists a neighborhood $W_{x}$ of $x$ in $U$ s.t.\ $g$ is transverse to $F_{x}$ at any $z\in W_{x}$. Then, $F_{1}$ and $W$ as above are obtained by the compactness of  $C(g,U,\phi)$.
\vskip2pt
\noindent
We orient $F_1$ and set $M=g^{-1}(F_1)$, hence $M$ is an orientable $C^1$-manifold in $E$ with ${\rm dim}(M)=m$. We then orient $M$ so that it is an oriented $(\Phi_0,g)$-preimage of $F_1$ (Definition \ref{fredorpreim}). Then $C(g,U,\phi)\subset M$ is compact even as a subset of $M$, and the following open covering of $C(g,U,\phi)$ exists:

\begin{lemma}\label{cover}
Let $(g,U,\phi)$ be an admissible triple satisfying \eqref{fin}, $W$, $F_1$, and $M$ be defined as above. Then, there exist $k\in\N$, and bounded open sets $V_1,\ldots V_k\subset M$ s.t.\
\begin{enumroman}
\item\label{cover1} $\overline{V}_j\subset M$, $j=1,\ldots k$ (by $\overline{V}_j$ we denote the closure of $V_j$ in $E$);
\item\label{cover2} $\displaystyle C(g,U,\phi)\subset V:=\cup_{j=1}^k V_j$;
\item\label{cover3} $\overline{V}_j$ is diffeomorphic to a closed convex subset of $\R^m$, $j=1,\ldots k$.
\end{enumroman}
\end{lemma}

\noindent
By \ref{cover3}, $\overline{V}_1,\ldots \overline{V}_k,\overline{V}$ are compact $ANR$-spaces. So, Lemma \ref{jmapsuff} implies that $\restr{\phi}{\overline{V}}\in CJ(\overline{V},F_1)$. Thus, by Proposition \ref{cjmapappr}, $\restr{\phi}{\overline{V}}$ is approximable. Recalling also that $C(g,U,\phi)$ and $\phi(\partial V)$ are compact sets (since $\phi$ is u.s.c.), we can find $\eps>0$ s.t.\ for all $f\in B_\eps(\restr{\phi}{\overline{V}})$ we have
\[{\rm dist}\big(0,(g-f)(\partial V)\big)>0.\]
So, Brouwer's degree for the triple $(\restr{g}{\overline{V}}-f,V,0)$ is well defined and it enjoys the reduction property displayed in Proposition \ref{reduction}. Now we prove that such degree is invariant:

\begin{lemma}\label{invfin}
Let $(g,U,\phi)$ be an admissible triple satisfying \eqref{fin}, $F_1$, $V$, $f$ be defined as above. Then, ${\rm deg}_B(\restr{g}{\overline{V}}-f,V,0)$ does not depend on $F_1$, $V$, and $f$.
\end{lemma}
\begin{proof}
We prove our assertion in three steps (backward):
\begin{itemize}[leftmargin=1cm]
\item[$(a)$] Let $F_1$, $V$ be fixed, and $f',f''\in B_\eps(\restr{\phi}{\overline{V}})$ be two approximations of $\phi$. By homotopy invariance of Brouwer's degree and Proposition \ref{cjmapappr}, by reducing $\eps>0$ if necessary we can apply \cite[Lemma 3.4]{OZZ} and get
\[{\rm deg}_B(\restr{g}{\overline{V}}-f',V,0)={\rm deg}_B(\restr{g}{\overline{V}}-f'',V,0).\]
\item[$(b)$] Let $F_1$ be fixed, $V',V''\subset M$ be open s.t.\ $C(g,U,\phi)\subset V'\cap V''$ and $\overline{V'}$, $\overline{V''}$ are compact $ANR$-spaces. Without loss of generality we may assume $V'\subset V''$. By Proposition \ref{apprusc} \ref{apprusc1}, by reducing $\eps>0$ if necessary we can find $f\in B_\eps(\restr{\phi}{\overline{V''}})$ s.t.\ $\restr{f}{\overline{V'}}\in B_\eps(\restr{\phi}{\overline{V'}})$. So, by the excision property of Brouwer's degree, we have
\[{\rm deg}_B(\restr{g}{\overline{V'}}-\restr{f}{\overline{V'}},V',0)={\rm deg}_B(\restr{g}{\overline{V''}}-f,V'',0).\]
\item[$(c)$] Finally, let $F'_1$, $F''_1$ be finite-dimensional subspaces of $F$, transverse to $g$ in $W$, s.t.\ $\phi(U)\subset F'_1\cap F''_1$. Then, by Proposition \ref{reduction} we have for any choice of $V$, $f$ the same ${\rm deg}_B(\restr{g}{\overline{V}}-f,V,0)$.
\end{itemize}
So, ${\rm deg}_B(\restr{g}{\overline{V}}-f,V,0)$ is independent of $F_1$, $V$, and $f$.
\end{proof}

\noindent
By virtue of Lemma \ref{invfin}, we can define a degree for the triple $(g,U,\phi)$:

\begin{definition}\label{degfin}
Let $(g,U,\phi)$ be an admissible triple satisfying \eqref{fin}, $F_1$, $V$, $f$ be defined as above. The \emph{degree} of $(g,U,\phi)$ is defined by
\[{\rm deg}(g,U,\phi)={\rm deg}_B(\restr{g}{\overline{V}}-f,V,0).\]
\end{definition}

\noindent
The following is a special homotopy invariance result, which will be useful in the forthcoming construction:

\begin{lemma}\label{homo}
Let $U\subseteq E$ be open, $h:U\times[0,1]\to F$ be an oriented $\Phi_0$-homotopy, $\phi\in CJ(U\times[0,1],F)$ be locally compact s.t.\
\begin{enumroman}
\item\label{homo1} the coincidence set
\[C(h,U\times[0,1],\phi)=\big\{(x,t)\in U\times[0,1]:\,h(x,t)\in\phi(x,t)\big\}\]
is compact;
\item\label{homo2} ${\rm dim}(\phi(U\times[0,1]))<\infty$.
\end{enumroman}
Then, the map $t\mapsto{\rm deg}(h(\cdot,t),U,\phi(\cdot,t))$ is constant in $[0,1]$.
\end{lemma}
\begin{proof}
By \ref{homo1}, \ref{homo2} we can find an open neighborhood $W\subset U\times[0,1]$ of $C(h,U\times[0,1],\phi)$ and a subspace $F_1\subseteq F$ s.t.\ ${\rm dim}(F_1)=m<\infty$, $\phi(U\times[0,1])\subset F_1$, and for all $t\in[0,1]$ $F_1$ is transverse to $h(\cdot,t)$ in the set
\[W_t:=\big\{x\in U:\,(x,t)\in W\big\}.\]
Set $M_1=h^{-1}(F_1)\cap W$, then $M_1$ is a $(m+1)$-dimensional $C^1$-mainfold in $E\times\R$ with boundary
\[\partial M_1=\big\{(x,t)\in M_1:\,t=0,1\big\}.\]
We orient $F_1$, so that the orientations of $h$, $F_1$ induce an orientation of $M_1$ in a unique way (Proposition \ref{transport}). Now let $V\subset M_1$ be an open (in $M_1$) neighborhood of $C(h,U\times[0,1],\phi)$, s.t.\ $\overline{V}\subset M_1$ is a compact $ANR$-space (the construction is analogous to that of Lemma \ref{cover}). By Propositions \ref{apprusc}, \ref{cjmapappr} the restriction $\restr{\phi}{\overline{V}}\in CJ(\overline{V},F_1)$ is approximable, and for all $\eps>0$ small enough we can find $f\in B_\eps(\restr{\phi}{\overline{V}})$ s.t.\ for all $t\in[0,1]$
\[{\rm deg}(h(\cdot,t),U,\phi(\cdot,t))={\rm deg}_B(\restr{h(\cdot,t)}{\overline{V}}-f(\cdot,t),V,0)\]
(Definition \ref{degfin}). By homotopy invariance of Brouwer's degree, the latter does not depend on $t\in[0,1]$, which concludes the proof.
\end{proof}

\noindent
Now we remove assumption \eqref{fin}. Let $(g,U,\phi)$ be an admissible triple, not necessarily satisfying \eqref{fin}. Since $g$ is locally proper, $\phi$ is locally compact, and $C(g,U,\phi)$ is compact (Definition \ref{triple}), we can find a bounded open neighborhood $U_1\subset U$ of $C(g,U,\phi)$ s.t.\ $\restr{g}{\overline{U}_1}$ is proper and $\restr{\phi}{\overline{U}_1}$ is compact. It is easily seen that $g-\phi:\overline{U}_1\to 2^F$ is a closed set-valued map, and $0\notin(g-\phi)(\partial U_1)$. Since $(g-\phi)(\partial U_1)$ is closed, there exists $\delta>0$ s.t.\
\[B_\delta(0)\cap(g-\phi)(\partial U_1)=\emptyset.\]
The set $K=\overline{\phi(\overline{U}_1)}$ is compact. So we can find a finite-dimensional subspace $F_1\subset F$ and a (single-valued) map $j_\delta\in C(K,F_1)$ s.t.\ for all $x\in K$
\[\|j_\delta(x)-x\|_F<\frac{\delta}{2}\]
(this is a classical result in nonlinear functional analysis, see e.g.\ \cite[Proposition 8.1]{D}). Set $\phi_1=j_\delta\circ\phi\in CJ(\overline{U}_1,F)$ (Definition \ref{cjmap}), moreover it satisfies
\[B_{\delta/2}(0)\cap\phi_1(\partial U_1)=\emptyset,\]
and $C(g,U_1,\phi_1)$ is compact. So, $(g,U_1,\phi_1)$ is an admissible triple satisfying \eqref{fin}. Definition \ref{degfin} then applies, and produces a degree ${\rm deg}(g,U_1,\phi_1)$. Moreover, such degree is invariant:

\begin{lemma}\label{inv}
Let $(g,U,\phi)$ be an admissible triple, $U_1$, $j_\delta$ be defined as above. Then, ${\rm deg}(g,U_1,\phi_1)$ does not depend on $U_1$, $j_\delta$.
\end{lemma}
\begin{proof}
Just as in Lemma \ref{invfin}, we divide the proof in two steps backward:
\begin{itemize}[leftmargin=1cm]
\item[$(a)$] Let $U_1$ be fixed, and $F_1$, $K$, $\delta$ be defined as above, and let $j'_\delta,j''_\delta\in C(K,F_1)$ be s.t.\ for all $x\in K$
\[\|j'_\delta(x)-x\|_F,\,\|j''_\delta(x)-x\|_F<\frac{\delta}{2}.\]
Set for all $(x,t)\in U_1\times[0,1]$
\[h(x,t)=g(x),\ \tilde\phi(x,t)=(1-t)j'_\delta(\phi(x))+tj''_\delta(\phi(x)).\]
Then $h:U_1\times[0,1]\to F$ is a $\Phi_0$-homotopy (Definition \ref{fredhom}). A more delicate question is proving that $\tilde\phi\in CJ(U_1\times[0,1],F)$, since this map is not explicitly defined as a composition of a $J$-map and a continuous single-valued function (Definition \ref{cjmap}). Since $\phi\in CJ(U,F)$, there exist a metric space $Z$, $\psi\in J(U_1,Z)$, and $k\in C(Z,F)$, s.t.\ $\phi=k\circ\psi$. Set for all $(x,t)\in U_1\times[0,1]$
\[\tilde\psi(x,t)=\psi(x)\times\{t\},\]
so clearly $\tilde\psi\in J(U_1\times[0,1],Z\times[0,1])$. Then set for all $(z,t)\in Z\times[0,1]$
\[\tilde k(z,t)=(1-t)j'_\delta(k(z))+tj''_\delta(k(z)),\]
so $\tilde k\in C(Z\times[0,1],F)$. Then,
\[\tilde\phi=\tilde k\circ\tilde\psi\in CJ(U_1\times[0,1],F).\]
Now we prove that the coincidence set
\[C(h,U_1\times[0,1],\tilde\phi)=\big\{(x,t)\in U_1\times[0,1]:\,h(x,t)\in\tilde\phi(x,t)\big\}\]
is compact. Let $(x_n,t_n)$ be a sequence in $C(h,U_1\times[0,1],\tilde\phi)$. Passing to a subsequence, we have $t_n\to t$. For all $n\in\N$ there exist $y'_n,y''_n\in\phi(x_n)$ s.t.\
\[g(x_n)=(1-t_n)j'_\delta(y'_n)+t_n j''_\delta(y''_n).\]
By compactness of $\restr{\phi}{\overline{U}_1}$, passing again to a subsequence we have $y'_n\to y'$, $y''_n\to y''$, which implies
\[g(x_n)\to (1-t)j'_\delta(y')+tj''_\delta(y'').\]
By properness of $\restr{g}{\overline{U}_1}$, we can find $x\in\overline{U}_1$ s.t.\ up to a further subsequence $x_n\to x$. We need to prove that $x\in U_1$. Arguing by contradiction, let $x\in\partial U_1$. Then, by the choice of $\delta>0$ we have
\[{\rm dist}(g(x),\phi(x))\ge\delta.\]
Besides, since $\phi$ is u.s.c.\ we have $y',y''\in\phi(x)$, hence by the metric properties of the maps $j'_\delta$, $j''_\delta$ we have
\begin{align*}
{\rm dist}(g(x),\phi(x)) &\le (1-t)\,{\rm dist}(j'_\delta(y'),\phi(x))+t\,{\rm dist}(j''_\delta(y''),\phi(x)) \\
&\le (1-t)\|j'_\delta(y')-y'\|_F+t\|j''_\delta(y'')-y''\|_F \le \frac{\delta}{2},
\end{align*}
a contradiction. So, $x\in U_1$ and we deduce that $C(h,U_1\times[0,1],\tilde\phi)$ is compact. Moreover, $\tilde\phi$ has a finite-dimensional rank. Then, by Lemma \ref{homo}, ${\rm deg}(h(\cdot,t),U_1,\tilde\phi(\cdot,t))$ is independent of $t\in[0,1]$. In particular, taking $t=0,1$ we get
\[{\rm deg}(g,U_1,j'_\delta\circ\phi)={\rm deg}(g,U_1,j''_\delta\circ\phi).\]
\item[$(b)$] Let $U'_1,U''_1\subset U$ be open neighborhoods of $C(g,U,\phi)$ s.t.\ the restrictions of $g$ to $\overline{U'}_1$, $\overline{U''}_1$ are proper and the restrictions of $\phi$ to $\overline{U'}_1$, $\overline{U''}_1$ are compact, respectively. Without loss of generality we may assume $U'_1\subseteq U''_1$, hence we continue the construction in $U''_1$. Then, independence of the degree follows from the excision property of Brouwer's degree.
\end{itemize}
So, ${\rm deg}(g,U_1,\phi_1)$ does not depend on the choice of $U_1$, $j_\delta$.
\end{proof}

\noindent
By virtue of Lemma \ref{inv}, we can define a degree for the triple $(g,U,\phi)$ extending Definition \ref{degfin}:

\begin{definition}\label{deg}
Let $(g,U,\phi)$ be an admissible triple, and $U_1$, $\phi_1$ be defined as above. The \emph{degree} of $(g,U,\phi)$ is defined by
\[{\rm deg}(g,U,\phi)={\rm deg}(g,U_1,\phi_1).\]
\end{definition}

\noindent
The degree theory we just introduced enjoys some classical properties:

\begin{proposition}\label{degprop}
The following properties hold:
\begin{enumroman}
\item\label{degprop1} (normalization) if $U\subset E$ is open s.t.\ $0\in U$, and $I$ is the naturally oriented identity of $E$, then
\[{\rm deg}(I,U,0)=1;\]
\item\label{degprop2} (domain additivity) if $(g,U,\phi)$ is an admissible triple, $U_1,\,U_2\subset U$ are open s.t.\ $U_1\cap U_2=\emptyset$, $C(g,U,\phi)\subset U_1\cup U_2$, then $(g,U_1,\phi)$, $(g,U_2,\phi)$ are admissible triples and
\[{\rm deg}(g,U,\phi)={\rm deg}(g,U_1,\phi)+{\rm deg}(g,U_2,\phi);\]
\item\label{degprop3} (homotopy invariance) if $U\subset E$ is open, $h:U\times[0,1]\to F$ is an oriented $\Phi_0$-homotopy, $\phi\in CJ(U\times[0,1],F)$ is locally compact s.t.\ $C(h,U\times[0,1],\phi)$ is compact, then for all $t\in[0,1]$ $(h(\cdot,t),U,\phi(\cdot,t))$ is an admissible triple and the function
\[t\mapsto{\rm deg}(h(\cdot,t),U,\phi(\cdot,t))\]
is constant in $[0,1]$.
\end{enumroman}
\end{proposition}
\begin{proof}
Properties \ref{degprop1}, \ref{degprop2} follow from Definition \ref{deg} and the corresponding properties of Brouwer's degree (the proof is straightforward, so we omit it).
\vskip2pt
\noindent
To prove \ref{degprop3}, we first fix $t\in[0,1]$. By compactness, we can find $\sigma>0$ and a bounded open neighborhood $W\subset U$ of the section
\[C_t=\big\{x\in U:\,(x,t)\in C(h,U\times[0,1],\phi)\big\}\]
s.t.\ $\restr{h}{\overline{W}\times I_\sigma}$ is proper and $\restr{\phi}{\overline{W}\times I_\sigma}$ is compact, where we have set $I_\sigma=[t-\sigma,t+\sigma]\cap [0,1]$. We also introduce a finite rank map $j\in C(K,F)$, close enough to the identity of $K=\overline{\phi(\overline{W}\times I_\sigma)}$. By the excision property of Browuer's degree and the construction above, for all $s\in I_\sigma$ we have
\[{\rm deg}(h(\cdot,s),U,\phi(\cdot,s))={\rm deg}(h(\cdot,s),W,j\circ\phi(\cdot,s)).\]
Besides, by Lemma \ref{homo} the function
\[s\mapsto\,{\rm deg}(h(\cdot,s),W,j\circ\phi(\cdot,s))\]
is constant in $I_\sigma$, hence ${\rm deg}(h(\cdot,s),U,\phi(\cdot,s))$ turns out to be locally constant in $[0,1]$. Since $[0,1]$ is connected, we get the conclusion.
\end{proof}

\begin{remark}\label{ghi}
In fact, Proposition \ref{degprop} \ref{degprop3} holds in a stronger form, i.e., for subsets of $E\times\R$ which are not necessarily products, as it can be seen from the proof. This is called {\em generalized homotopy invariance}.
\end{remark}

\section{Persistence results and bifurcation points}\label{sec6}

\noindent
We can now prove the main results of the present paper, as announced in the Introduction. Throughout this section, $E$ and $F$ are two real Banach spaces, $\Omega\subset E$ is an open (not necessarily bounded) set s.t.\ $0\in\Omega$, $L\in\Phi_0(E,F)$ satisfy ${\rm ker}\,L\neq 0$, $C\in\mathcal{L}(E,F)$ is another bounded linear operator, and $\phi\in CJ(\overline\Omega,F)$ is locally compact. The linear operators $L$, $C$ satisfy the transversality condition \eqref{trans}. For all $\eps,\lambda\in\R$ we consider the perturbed problem \eqref{pev}, whose solutions are meant in the following sense:

\begin{definition}\label{sol}
A \emph{solution} of \eqref{pev} is a triple $(x,\eps,\lambda)\in\partial\Omega\times\R\times\R$ s.t.\
\[Lx-\lambda Cx+\eps\phi(x)\ni 0.\]
The set of solutions is denoted by $\mathcal{S}$. Moreover a solution $(x,\eps,\lambda)\in\mathcal{S}$ is a \emph{trivial} solution, if $\eps=\lambda=0$. Finally, we say that $x_0\in\partial\Omega$ is a \emph{bifurcation point}, if $(x_0,0,0)\in\mathcal{S}$ and any neighborhood of $(x_0,0,0)$ in $E\times\R\times\R$ contains at least one non-trivial solution.
\end{definition}

\noindent
Clearly, any trivial solution $(x,0,0)$ of \eqref{pev} identifies with its first component $x$. The set of such vectors is
\[\mathcal{S}_0=\partial\Omega\cap{\rm ker}\,L.\]
Regarding our definition of a bifurcation point, we note that it is analogous to that of \cite{BCFP}, and fits in the very general definition given in \cite[p.\ 2]{CH}. Finally, we note that, whenever $(x,0,\lambda)\in\mathcal{S}$, $(x,\lambda)$ is an eigenpair of the eigenvalue problem \eqref{ev}: thus, we keep the names {\em eigenvector} for $x$ and {\em eigenvalue} for $\lambda$, respectively, for any triple $(x,\eps,\lambda)\in\mathcal{S}$.
\vskip2pt
\noindent
As observed in \cite[Remark 5.1]{BCFP}, transversality condition \eqref{trans} is in fact equivalent to
\[{\rm im}\,L\oplus C({\rm ker}\,L)=F.\]
Thus, we can find $b>0$ s.t.\ $L-\lambda C\in\Phi_0(E,F)$ is invertible for all $0<|\lambda|\le b$. Moreover, since $0\notin\partial\Omega$, for any bifurcation point $x_0\in\mathcal{S}_0$ we can find a neighborhood $W\subset E\times\R\times\R$ of $(x_0,0,0)$ s.t.\ any triple $(x,\eps,\lambda)\in\mathcal{S}\cap W$ actually must have $\eps\neq 0$.
\vskip2pt
\noindent
The map $\lambda\mapsto L-\lambda C$ (which is orientable according to Definition \ref{orientcont} since its domain is simply connected, see Proposition \ref{fredmaptop} \ref{fredmaptop3}) exhibits a sign jump property (a special case of \cite[Corollary 5.1]{BFPS}):

\begin{lemma}\label{signjump}
Let $b>0$ be defined as above, $h\in C([-b,b],\Phi_0(E,F))$ be defined by
\[h(\lambda)=L-\lambda C,\]
and oriented. Then:
\begin{enumroman}
\item\label{signjump1} the map $\lambda\mapsto\,{\rm sign}\,h(\lambda)$ is constant in both $[-b,0)$ and $(0,b]$;
\item\label{signjump2} ${\rm sign}\,h(b)\neq{\rm sign}\,h(-b)$ iff ${\rm dim}({\rm ker}\,L)$ is odd.
\end{enumroman}
\end{lemma}

\noindent
Lemma \ref{signjump} above is the reason why the assumption that ${\rm dim}({\rm ker}\,L)$ is odd is so important in our theory. Now we prove an existence result on {\em bounded} subdomains, which is the core of our argument:

\begin{proposition}\label{exi}
Let ${\rm dim}({\rm ker}\,L)$ be odd, \eqref{trans} hold, and $U\subseteq\Omega$ be an open, bounded set s.t.\ $0\in U$ and $\restr{\phi}{\overline{U}}\in CJ(\overline{U},F)$ is compact. Then, there exist $a,b>0$ s.t.\ for all $\eps\in[-a,a]$ there exist $\lambda\in[-b,b]$, $x\in\partial U$ s.t.\
\[Lx-\lambda Cx+\eps\phi(x)\ni 0.\]
\end{proposition}
\begin{figure}
\centering
\begin{tikzpicture}[scale=2]
\draw (0,0) node[below]{$(-a, -b)$}; 
\draw (3,0) node[below]{$(a, -b)$};
\draw (0,2) node[above]{$(-a, b)$};
\draw (3,2) node[above]{$(a, b)$};
\draw (1.5, 0) node[below]{$(0, -b)$};
\draw (1.5, 2) node[above]{$(0, b)$};
\draw (2.7, 1.7) node{$\mathcal{R}$};
\draw[clip] (0,0) rectangle (3, 2); 
\draw[dashed] (1.5, 0) -- (1.5, 2);
\draw[line width=2pt] (0, 1) .. controls (2,2) and ( 1, 0.5 ) .. (3, 0.8); 
\draw (2.4,1) node{$\Gamma$};
\end{tikzpicture}
\caption{The set $\Gamma$ cutting the rectangle $\mathcal{R}$.}
\label{fig1}
\end{figure}
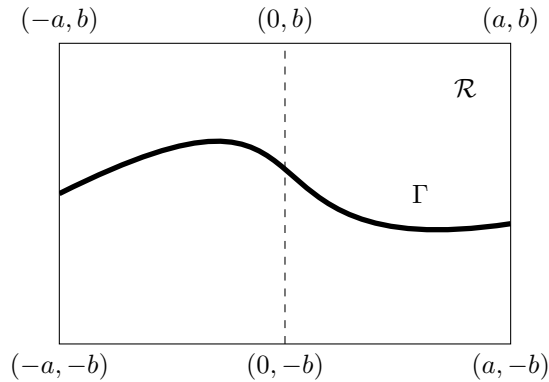
\begin{proof}
Let $b>0$ be as in Lemma \ref{signjump}, and fix $a>0$ (to be better determined later). Set
\[\mathcal{R}=[-a,a]\times[-b,b],\]
and define the set
\beq\label{cappa}
\mathcal{K}=\big\{(x,\eps,\lambda)\in\partial U\times\mathcal{R}:\,Lx-\lambda Cx+\eps\phi(x)\ni 0\big\}.
\eeq
The set $\mathcal{K}\subset E\times\R\times\R$ is compact. Indeed, let $(x_n,\eps_n,\lambda_n)$ be a sequence in $\mathcal{K}$. Then, $(\eps_n,\lambda_n)$ is a bounded sequence in $\mathcal{R}$, hence passing to a subsequence we have $(\eps_n,\lambda_n)\to(\eps,\lambda)$ for some $(\eps,\lambda)\in\mathcal{R}$. As seen above, we have eventually $\eps_{n}\neq 0$. Now set for all $n\in\N$
\[y_n=-\frac{1}{\eps_n}(Lx_n-\lambda_n Cx_n)\in\phi(x_n).\]
Since $\overline{\phi(\overline{U})}$ is compact, passing if necessary to a further subsequence, we have $y_n\to y$ for some $y\in F$, which implies
\[\lim_n(Lx_n-\lambda Cx_n)=\lim_n(Lx_n-\lambda_n Cx_n)+\lim_n(\lambda_n-\lambda)Cx_n=-\eps y\]
(recall that $(x_n)$ is a bounded sequence and $C$ is a bounded operator). The operator $L-\lambda C\in\Phi_0(E,F)$ is proper on closed and bounded subsets of $E$, hence passing again to a subsequence if necessary we have $x_n\to x$ for some $x\in\partial U$. Thus, $(x_n,\eps_n,\lambda_n)\to(x,\eps,\lambda)$ for some $(x,\eps,\lambda)\in\mathcal{K}$.
\vskip2pt
\noindent
Clearly, the projection of $\mathcal{K}$ onto $\mathcal{R}$, namely the set
\[\Gamma=\big\{(\eps,\lambda)\in\mathcal{R}:\,(x,\eps,\lambda)\in\mathcal{K} \ \text{for some $x\in\partial U$}\big\},\]
is compact as well. Now we choose an orientation of $L\in\Phi_0(E,F)$ (Definition \ref{fredorient}), and fix $(\eps,\lambda)\in\mathcal{R}\setminus\Gamma$. Then $(L-\lambda C,U,-\eps\phi)$ is an admissible triple (Definition \ref{triple}), since the coincidence set
\[C(L-\lambda C,U,-\eps\phi)=\big\{x\in U:\,Lx-\lambda Cx+\eps\phi(x)\ni 0\big\}\]
is compact. Indeed, arguing as above, for any sequence $(x_n)$ in $C(L-\lambda C,U,-\eps\phi)$ we can find a relabeled subsequence s.t.\ $x_n\to x$ for some $x\in\overline{U}$. It remains to prove that $x\in U$. Otherwise, we would have $x\in\partial U$, hence $(\eps,\lambda)\in\Gamma$, a contradiction.
\vskip2pt
\noindent
So, the integer-valued map
\[(\eps,\lambda)\mapsto{\rm deg}(L-\lambda C,U,-\eps\phi)\]
is well defined in the relatively open set $\mathcal{R}\setminus\Gamma$ (Definition \ref{deg}), and constant on any connected component of $\mathcal{R}\setminus\Gamma$ by homotopy invariance (Proposition \ref{degprop} \ref{degprop3}).
\vskip2pt
\noindent
By the choice of $b>0$, both operators $L\pm bC$ are invertible. Hence, $(0,\pm b)\in\mathcal{R}\setminus\Gamma$ (recall that $0\notin\partial U$). We claim that
\beq\label{degjump}
{\rm deg}(L+bC,U,0)\neq{\rm deg}(L-bC,U,0).
\eeq
Indeed, let $L+bC\in\Phi_0(E,F)$ be naturally oriented, then
\[{\rm sign}(L+bC)=1\]
(Definition \ref{fredorient} \ref{fredorient2}, \ref{fredorient3}). We fix a non-trivial, finite-dimensional subspace $F_1\subset F$ and set $E_1=(L+bC)^{-1}(F_1)$, then we orient $F_1$ and $E_1$ so that $E_1$ is the oriented $(L+bC)$-preimage of $F_1$. With such an orientation of the involved spaces and maps, recalling Definitions \ref{deg}, \ref{degfin}, and \cite[p.\ 121]{H}, we have
\[{\rm deg}(L+bC,U,0)={\rm deg}_B(\restr{(L+bC)}{E_1},U\cap E_1,0)=1.\]
Since ${\rm dim}({\rm ker}\,L)$ is odd, by Lemma \ref{signjump} \ref{signjump2} we have
\[{\rm sign}(L-bC)=-1,\]
which, repeating the construction above with the same orientations, leads to
\[{\rm deg}(L-bC,U,0)=-1.\]
Similar arguments can be developed if different orientations are chosen, so \eqref{degjump} holds in any case.
\vskip2pt
\noindent
By \eqref{degjump}, we deduce that $(0,\pm b)$ lie in different connected components of $\mathcal{R}\setminus\Gamma$. By reducing further $a>0$ if necessary, we may assume that $(\eps,\pm b)$ lie in different connected components of $\mathcal{R}\setminus\Gamma$, for all $\eps\in[-a,a]$ (the situation is depicted in figure \ref{fig1}). So, for all $\eps\in[-a,a]$ we can find $\lambda\in[-b,b]$ s.t.\ $(\eps,\lambda)\in\Gamma$, which concludes the proof.
\end{proof}

\noindent
Proposition \ref{exi} is the main tool for proving persistence of the eigenpairs under a set-valued perturbation, with the additional assumption that the set $\Omega_0:=\overline\Omega\cap {\rm ker} L$ is {\em compact} (note that $\Omega_0\neq\emptyset$ as $0\in\Omega$). We begin with eigenvalues:

\begin{theorem}\label{perval}
Let ${\rm dim}({\rm ker}\,L)$ be odd, \eqref{trans} hold, and $\Omega_0:=\overline\Omega\cap {\rm ker} L$ be non-empty and compact. Then, for all $c>0$ small enough there exist $a,b>0$ s.t.\ the set-valued map $\Gamma:[-a,a]\to 2^{[-b,b]}$ defined by
\[\Gamma(\eps)=\big\{\lambda\in[-b,b]:\,(x,\eps,\lambda)\in\mathcal{S} \ \text{for some $x\in\partial\Omega\cap B_c(\mathcal{S}_0)$}\big\}\]
has the following properties:
\begin{enumroman}
\item\label{perval1} $\Gamma(\eps)\neq\emptyset$ for all $\eps\in[-a,a]$;
\item\label{perval2} $\Gamma$ is u.s.c.
\end{enumroman}
\end{theorem}
\begin{proof}
Since the set $ \Omega_0$ is compact, we can find a bounded open neighborhood $W\subset E$ of $\Omega_0$ s.t.\ $\restr{\phi}{\overline{U}}$ is compact, where we have set $U=W\cap\Omega$. Clearly, $U$ is a bounded open set s.t.\ $0\in U$. Then we can apply Proposition \ref{exi} and thus find a rectangle
\[\mathcal{R}=[-a,a]\times[-b,b] \ (a,b>0),\]
s.t.\ for all $\eps\in[-a,a]$ there exist $\lambda\in[-b,b]$, $x\in\partial U$ s.t.\
\[Lx-\lambda Cx+\eps\phi(x)\ni 0.\]
Besides, let $c>0$ be small enough that $B_c(\mathcal{S}_0)\subset W$ (recall that $\mathcal{S}_0=\partial\Omega\cap{\rm ker}\,L$). We define $\mathcal{K}\subset E\times\mathcal{R}$ as in \eqref{cappa}. As in the proof of Proposition \ref{exi}, we see that $\mathcal{K}$ is compact. We define a set-valued map $\psi:\mathcal{R}\to 2^F$ by setting
\[\psi(\eps,\lambda)=\big\{x\in\partial U:\,(x,\eps,\lambda)\in\mathcal{K}\big\}.\]
We claim that 
\[\psi(0,0)=\mathcal{S}_0\subset B_c(\mathcal{S}_0).\]
Indeed, for all $x\in\mathcal{S}_0$ we have $x\in\Omega_0\subset W$, which along with $x\in\partial\Omega$ implies $x\in\partial U$, while $(x,0,0)\in\mathcal{S}$, so $(x,0,0)\in\mathcal{K}$. Conversely, if $x\in\psi(0,0)$, then $x\in\partial U\subseteq\partial\Omega\cup\partial W$, while $x\in\Omega_0\subset W$, so we deduce $x\in\partial\Omega$ and since $(x,0,0)\in\mathcal{K}$ we have $x\in\mathcal{S}_0$.
\vskip2pt
\noindent
Moreover, the set ${\rm graph}\,\psi\subset\mathcal{R}\times E$ is obtained as a continuous image of $\mathcal{K}$ (by a swap of coordinates) and hence is compact. So, by Lemma \ref{compgr}, $\psi$ is u.s.c.
\vskip2pt
\noindent
Therefore by reducing $a,b>0$ if necessary we have for all $(\eps,\lambda)\in\mathcal{R}$
\beq\label{loc}
\psi(\eps,\lambda)\subset B_c(\mathcal{S}_0).
\eeq
Now we can prove both assertions. Fix $\eps\in[-a,a]$. By Proposition \ref{exi} there exist $\lambda\in[-b,b]$, $x\in\partial U$ s.t.\ $x\in\psi(\eps,\lambda)$, so by \eqref{loc} we have $x\in B_c(\mathcal{S}_0)$. Then, $x\in\partial U\cap W\subset\partial\Omega$, so $(x,\eps,\lambda)\in\mathcal{S}$. Thus $\lambda\in\Gamma(\eps)$, which proves \ref{perval1}.
\vskip2pt
\noindent
To prove \ref{perval2}, we just need to note that
\[{\rm graph}\,\Gamma=\big\{(\eps,\lambda)\in\mathcal{R}:\,(x,\eps,\lambda)\in\mathcal{K} \ \text{for some $x\in\partial\Omega\cap B_c(\mathcal{S}_0)$}\big\}\]
is but the projection of $\mathcal{K}$ onto $\mathcal{R}$, hence compact. Then Lemma \ref{compgr} ensures that $\Gamma:[-a,a]\to 2^{[-b,b]}$ is u.s.c.
\end{proof}

\noindent
A similar persistence result holds for the eigenvectors:

\begin{theorem}\label{pervec}
Let ${\rm dim}({\rm ker}\,L)$ be odd, \eqref{trans} hold, and $\Omega_0$ be compact. Then, for all $c>0$ small enough there exist $a,b>0$ s.t.\ the set-valued map $\Sigma:[-a,a]\to 2^E$ defined by
\[\Sigma(\eps)=\big\{x\in\partial\Omega\cap B_c(\mathcal{S}_0):\,(x,\eps,\lambda)\in\mathcal{S} \ \text{for some $\lambda\in[-b,b]$}\big\}\]
has the following properties:
\begin{enumroman}
\item\label{pervec1} $\Sigma(\eps)\neq\emptyset$ for all $\eps\in[-a,a]$;
\item\label{pervec2} $\Sigma$ is u.s.c.
\end{enumroman}
\end{theorem}
\begin{proof}
As in the proof of Theorem \ref{perval}, for all $c>0$ small enough we find a rectangle $\mathcal{R}=[-a,a]\times[-b,b]$ ($a,b>0$) and an open neighborhood $W\subset E$ of $\Omega_0$ s.t., setting $U=W\cap\Omega$, the set $\mathcal{K}$ defined by \eqref{cappa} is compact, and moreover $x\in B_c(\mathcal{S}_0)$ whenever $(x,\eps,\lambda)\in\mathcal{K}$ (see \eqref{loc}).
\vskip2pt
\noindent
In particular, for all $(x,\eps,\lambda)\in\mathcal{K}$ we have $x\in\Sigma(\eps)$. Then, Proposition \ref{exi} implies \ref{pervec1}. Moreover, since
\[{\rm graph}\,\Sigma=\big\{(\eps,x)\in[-a,a]\times(\partial\Omega\cap B_c(\mathcal{S}_0)):\,(x,\eps,\lambda)\in\mathcal{K} \ \text{for some $\lambda\in[-b,b]$}\big\}\]
is the projection of $\mathcal{K}$ onto $[-a,a]\times E$, hence compact, by Lemma \ref{compgr} we also deduce \ref{pervec2}.
\end{proof}

\noindent
As a consequence, we prove that the set $\mathcal{S}_0$ contains at least one bifurcation point (Definition \ref{sol}):

\begin{theorem}\label{bif}
Let ${\rm dim}({\rm ker}\,L)$ be odd, \eqref{trans} hold, and $\Omega_0$ be compact. Then, problem \eqref{pev} has at least one bifurcation point.
\end{theorem}
\begin{proof}
We argue by contradiction: assume that $\mathcal{S}_0$ contains no bifurcation points, i.e., for all $x\in\mathcal{S}_0$ there exists an open neighborhood $\mathcal{U}_x\subset E\times\R\times\R$ of $(x,0,0)$, s.t.\ for all $(x,\eps,\lambda)\in\mathcal{S}\cap\mathcal{U}_x$ we have $(\eps,\lambda)=(0,0)$. The family $(\mathcal{U}_x)_{x\in\mathcal{S}_0}$ is an open covering of the compact set $\mathcal{S}_0\times\{(0,0)\}$ in $E\times\R\times\R$, so we can find a finite sub-covering, which we relabel as $(\mathcal{U}_i)_{i=1}^m$.
\vskip2pt
\noindent
Let $a,b,c>0$ be s.t.\
\[B_c(\mathcal{S}_0)\times\mathcal{R}\subset\bigcup_{i=1}^m\mathcal{U}_i,\]
where as usual $\mathcal{R}=[-a,a]\times[-b,b]$. Thus, we have
\[\mathcal{S}\cap\big(B_c(\mathcal{S}_0)\times\mathcal{R}\big)=\mathcal{S}_0\times\{(0,0)\}\]
(i.e., there are no solutions in $B_c(\mathcal{S}_0)\times\mathcal{R}$ except the trivial ones). By reducing $a,b,c>0$ if necessary, Theorem \ref{pervec} applies. So, for all $\eps\in[-a,a]\setminus\{0\}$ there exist $x\in\partial\Omega\cap B_c(\mathcal{S}_0)$, $\lambda\in[-b,b]$ s.t.\ $(x,\eps,\lambda)\in\mathcal{S}$, a contradiction.
\end{proof}

\begin{remark}
\label{norma-omega}
Since ${\rm ker}\,L$ has finite dimension, compactness of $\Omega_0$ (which is assumed in the statements of the last theorems) is clearly verified as long as $\Omega$ is bounded. On the other hand, trivial examples in Euclidean spaces show that, if $\Omega$ is unbounded, then $\Omega_0$ may fail to be compact. We want to present a special type of (possibly unbounded) domains which satisfy our assumption: let $\gamma:E\to\R$ be a continuous norm and set
\[\Omega=\big\{x\in E:\,\gamma(x)<1\big\}.\]
Let $(x_n)$ be a sequence in $\Omega_0=\overline\Omega\cap{\rm ker}\,L$. Without loss of generality we may assume $x_n\neq 0$ for all $n\in\N$. Setting $y_n=x_n/\|x_n\|$, we define a bounded sequence $(y_n)$ in the finite-dimensional space ${\rm ker}\,L$, so passing to a subsequence if necessary we have $y_n\to y$, $\|y\|=1$. By continuity, $\gamma(y_n)\to\gamma(y)>0$, so
\[\|x_n\|=\frac{\gamma(x_n)}{\gamma(y_n)}\le\frac{1}{\gamma(y_n)}\]
is bounded. Passing to a further subsequence, we have $\|x_n\|\to\mu\ge 0$, and hence $x_n\to\mu y$. So, $\Omega_0$ is compact.
\end{remark}

\noindent
We conclude this section by presenting a special case of Theorem \ref{bif}:

\begin{corollary}\label{sphere}
Let ${\rm dim}({\rm ker}\,L)$ be odd, \eqref{trans} hold, $\gamma\in C(E,\R)$ be a norm, and
\[\Omega=\big\{x\in E:\,\gamma(x)<1\big\}.\]
Then, problem \eqref{pev} has at least one bifurcation point.
\end{corollary}
\begin{proof}
Clearly $\Omega\subset E$ is an open set s.t.\ $0\in\Omega$ and $\partial\Omega=\gamma^{-1}(1)$. Moreover, by Remark \ref{norma-omega}, the set $\Omega_0=\overline\Omega\cap{\rm ker}\,L$ is compact. Thus, we can apply Theorem \ref{bif} and conclude.
\end{proof}

\section{Examples and applications}\label{sec7}

\noindent
We devote this final section to an application of our abstract results in the field of differential inclusions. We consider problem \eqref{odi} stated in the Introduction. We recall that $\Phi(u):[0,1]\to 2^\R$ is a set-valued mapping depending on $u$, to be defined later, while $\eps,\lambda\in\R$ are parameters and $\|\cdot\|_1$ is the usual $L^1$-norm on $[0,1]$. Problem \eqref{odi} falls into the general pattern \eqref{pev}, with the following definitions. Set
\[E=\big\{u\in C^2([0,1],\R):\,u'(0)=u'(1)=0\big\}, \ F=C^0([0,1],\R),\]
endowed with the usual norms. Then $E$, $F$ are real Banach spaces, in particular $E$ is a 2-codimensional subspace of $C^2([0,1],\R)$. Moreover, set for all $u\in E$
\[Lu=u''+u', \ Cu=u.\]
Then, $L,C\in\mathcal{L}(E,F)$. Moreover, $L\in\Phi_0(E,F)$ as the composition of the embedding $E\hookrightarrow C^2([0,1],\R)$ (which is Fredholm of index $-2$) and the linear differential operator $u\mapsto u''+u'$ (which is Fredholm of index 2 between $C^2([0,1],\R)$ and $F$). In order to check the transversality condition \eqref{trans}, we need more detailed information about $L$. It is easily seen that ${\rm ker}\,L$ is the space of constant functions, i.e.,
\[{\rm ker}\,L=\R,\]
in particular ${\rm dim}({\rm ker}\,L)=1$ (odd). Moreover, we have
\[{\rm im}\,L=\Big\{f\in F:\,\int_0^1f(t)e^t\,dt=0\Big\}.\]
Indeed, for all $f\in{\rm im}\,L$ there exists $u\in E$ s.t.\ $u''+u'=f$, so integrating by parts we deduce
\[\int_0^1f(t)e^t\,dt=\int_0^1u''(t)e^t\,dt+\int_0^1u'(t)e^t\,dt=0.\]
Besides, since $L\in\Phi_0(E,F)$, we have ${\rm dim}({\rm coker}\,L)=1$ (Definition \ref{fredop}), so the condition above is also sufficient. Now we prove \eqref{trans}, or equivalently
\[{\rm im}\,L\oplus C({\rm ker}\,L)=F.\]
Indeed, $C({\rm ker}\,L)=\R$ is not contained in the 1-codimensional subspace ${\rm im}\,L$, hence it is a (direct) complement for it in $F$.
\vskip2pt
\noindent
The integral constraint rephrases as $u\in\partial\Omega$, where we have set
\[\Omega=\big\{u\in E:\,\|u\|_1<1\big\}.\]
Since $\|\cdot\|_1$ is a continuous norm on $E$, $\Omega$ is an (unbounded) open set s.t.\ $\Omega_0=\overline\Omega\cap{\rm ker}\,L$ is compact (Remark \ref{norma-omega}). Moreover, from the characterization of ${\rm ker}\,L$ we have $\mathcal{S}_0=\{\pm 1\}$.
\vskip2pt
\noindent
The construction of $\Phi$ requires some care. We are going to consider a set-valued map $\phi\in CJ(\overline\Omega,F)$, and then set for all $u\in\overline\Omega$, $t\in[0,1]$
\[\Phi(u)(t)=\big\{w(t):\,w\in\phi(u)\big\}\]
(this can be seen as a set-valued superposition operator). Details will be given in Examples \ref{gasinski}, \ref{benedetti}, and \ref{zecca} below. We can now apply our abstract results to prove existence of a bifurcation point:

\begin{theorem}\label{bifodi}
Let $E$, $F$, $L$, $C$, $\Omega$, and $\Phi$ be as above, being $\phi\in CJ(\overline\Omega,F)$ locally compact. Then, there exist sequences $(u_n)$ in $\partial\Omega$, $(\eps_n)$ in $\R\setminus\{0\}$, $(\lambda_n)$ in $\R$ s.t.\ $(u_n,\eps_n,\lambda_n)$ is a solution of \eqref{odi} for all $n\in\N$, and
\[u_n\to\pm 1, \ \eps_n\to 0, \ \lambda_n\to 0.\]
\end{theorem}
\begin{proof}
By Remark \ref{norma-omega} and Corollary \ref{bif}, problem \eqref{pev} has at least one bifurcation point in $\mathcal{S}_0$, that is, either the constant $1$ or $-1$. So, we can find a sequence $(u_n,\eps_n,\lambda_n)$ if non-trivial solutions of \eqref{pev} (more precisely, with $\eps_n\neq 0$) converging to either $(1,0,0)$ or $(-1,0,0)$ in $E\times\R\times\R$. By the definition of $\Phi$, for all $n\in\N$ and all $t\in[0,1]$ we have
\[u''_n(t)+u'_n(t)-\lambda_n u_n(t)+\eps_n\Phi(u_n)(t)\ni 0,\]
so $(u_n,\eps_n,\lambda_n)$ solves \eqref{odi}.
\end{proof}

\noindent
We present three examples of locally compact $CJ$-maps $\phi:\overline\Omega\to 2^F$. The first and second examples are quite easy, $\phi$ being defined by means of finite-dimensional reduction.

\begin{example}\label{gasinski}
We define a set-valued map $\phi:\overline\Omega\to 2^F$ whose values consist of piecewise affine functions along a decomposition of $[0,1]$, satisfying some bounds at the nodal points. Fix $m\in\N$, points $0=t_0<t_1<\ldots<t_m=1$, and $\rho\in(0,1)$, then for all $u\in\overline\Omega$ set
\[\phi(u)=\big\{w\in F:\,\text{$w$ is affine in $[t_{j-1},t_j]$, $j=1,\ldots m$ and $u(t_j)-\rho\le w(t_j)\le u(t_j)+\rho$, $j=0,\ldots m$}\big\}.\]
We first prove that $\phi$ has convex values. For any $u\in\overline\Omega$, $w_0,w_1\in\phi(u)$, and $\mu\in[0,1]$ the function $w=(1-\mu)w_0+\mu w_1$ is affine in any interval $[t_{j-1},t_j]$ ($j=1,\ldots m$), and clearly
\[u(t_j)-\rho\le w(t_j)\le u(t_j)+\rho \ (j=0,\ldots m).\]
Then we prove that $\phi$ has compact values. Let $u\in\overline\Omega$, $(w_n)$ be a sequence in $\phi(u)$. Then we can find $\alpha_1,\ldots\alpha_m>0$ s.t.\ $|w'_n(t)|\le\alpha_j$ for all $t\in(t_{j-1},t_j)$, $j=1,\ldots m$, and $n\in\N$. So the sequence $(w_n)$ is uniformly bounded and equi-continuous, hence by Ascoli's theorem we can pass to a subsequence s.t.\ $w_n\to w$ in $F$. Due to uniform convergence, $w$ is piecewise affine and satisfies the bounds at $t_j$ ($j=0,\ldots m$), so $w\in\phi(u)$. Thus, $\phi(u)$ is compact.
\vskip2pt
\noindent
Moreover, the set-valued map $\phi$ is locally compact. Indeed, let $(u_n)$ be a bounded sequence in $\overline\Omega$ and $(w_n)$ be a sequence in $F$, s.t.\ $w_n\in\phi(u_n)$ for all $n\in\N$. Recalling that $(u'_n)$ is uniformly bounded, we can argue as above to find relabeled subsequence $w_n$ s.t.\ $w_n\to w$ in $F$. Thus, $\overline{\phi(u_n)}$ is compact.
\vskip2pt
\noindent
We prove finally that ${\rm graph}\,\phi$ is closed in $\overline\Omega\times F$. Indeed, let $(u_n,w_n)$ be a sequence in $\overline\Omega\times F$ s.t.\ $w_n\in\phi(u_n)$ for all $n\in\N$, and $(u_n,w_n)\to(u,w)$. Then, $w\in\phi(u)$. By \cite[Proposition 4.15]{G}, $\phi$ is u.s.c. We conclude that $\phi\in CJ(\overline\Omega,F)$ and is locally compact.
\end{example}

\begin{example}\label{benedetti}
In this second example, $\phi(u)$ depends on $u$ in a single-valued sense, but is multiplied by an interval depending on the mean value of $u$ (non-local dependence). Fix $f\in C^0(\R,\R)$, $\alpha,\beta:\R\to\R$ s.t.\ $\alpha$ is lower semi-continuous, $\beta$ is upper semi-continuous, and $\alpha(s)\le\beta(s)$ for all $s\in\R$. For all $u\in\overline\Omega$ set
\[\phi(u)=f(u)[\alpha(\bar u),\beta(\bar u)], \ \bar u=\int_0^1 u(\tau)\,d\tau.\]
Obviously, $\phi:\overline\Omega\to 2^F$ has convex values. We prove that $\phi$ has compact values. Let $u\in\overline\Omega$, $(w_n)$ be a sequence in $\phi(u)$. Then, for all $n\in\N$ there exists $c_n\in[\alpha(\bar u),\beta(\bar u)]$ s.t.\ $w_n=c_n f(u)$. The sequence $(c_n)$ is bounded, so passing to a subsequence we have $c_n\to c$ for some $c\in\R$. Set $w=cf(u)$, then clearly $w_n\to w$ in $F$ and $w\in\phi(u)$. Thus, $\phi(u)$ is compact.
\vskip2pt
\noindent
The map $\phi$ is locally compact. Indeed, let $(u_n)$ be a bounded sequence in $\overline\Omega$ and $(w_n)$ be a sequence in $F$, s.t.\ $w_n\in\phi(u_n)$ for all $n\in\N$. Then for all $n\in\N$ we can find $c_n\in[\alpha(\bar u_n),\beta(\bar u_n)]$ s.t.\ $w_n=c_nf(u_n)$. Since $(u_n)$ is uniformly bounded and equi-continuous, passing to a subsequence we have $u_n\to u$ uniformly in $[0,1]$ (note that $u\notin E$ in general). Hence, $\bar u_n\to\bar u$. So $(c_n)$ turns out to be bounded, and up to a subsequence $c_n\to c$. Passing to the limit, due to the properties of $\alpha$ and $\beta$, we have
\[\alpha(\bar u)\le c\le\beta(\bar u).\]
So, setting $w=cf(u)$, we deduce $w_n\to w$ in $F$. Thus, $\overline{\phi(u_n)}$ is compact.
\vskip2pt
\noindent
We see now that ${\rm graph}\,\phi$ is closed in $\overline\Omega\times F$. Indeed, let $(u_n,w_n)$ be a sequence in $\overline\Omega\times F$ s.t.\ $w_n\in\phi(u_n)$ for all $n\in\N$, and $(u_n,w_n)\to(u,w)$. For all $n\in\N$ we find $c_n\in[\alpha(\bar u_n),\beta(\bar u_n)]$ s.t.\ $w_n=c_nf(u_n)$. Then, $\bar u_n\to\bar u$, and $f(u_n)\to f(u)$ uniformly in $[0,1]$. We prove now that $(c_n)$ converges, indeed avoiding trivial cases we may assume that $f(u(t))\neq 0$ at some $t\in[0,1]$, then
\[\lim_n c_n=\frac{w_n(t)}{f(u_n(t))}=\frac{w(t)}{f(u(t))}=c,\]
with $c\in[\alpha(\bar u),\beta(\bar u)]$. So $w\in\phi(u)$. By \cite[Proposition 4.15]{G} again, $\phi$ is u.s.c. We conclude that $\phi\in CJ(\overline\Omega,F)$ and is locally compact.
\end{example}

\noindent
The last example is more sophisticated, since in the construction of $\phi$ we preserve the infinite dimension, and we apply some classical results from functional analysis to prove all required compactness properties. We recall such results, starting from a weak notion of compactness in $L^1$ (see \cite[Definition 2.94]{GHO}, \cite[Definition 4.2.1]{KOZ}):

\begin{definition}\label{semicompact}
A sequence $(v_n)$ in $L^1([0,1],\R)$ is said to be semicompact, if
\begin{enumroman}
\item\label{semicompact1} it is integrably bounded, i.e., if there exists $g\in L^1(0,1)$ s.t.\ $|v_n(t)|\le g(t)$ for a.e.\ $t\in[0,1]$ and all $n\in\N$;
\item\label{semicompact2} the image sequence $(v_n(t))$ is relatively compact in $\R$ for a.e.\ $t \in [0,1]$.
\end{enumroman}
\end{definition}

\noindent
The following result follows from the Dunford-Pettis Theorem (see also \cite[Proposition 4.21]{KOZ}):

\begin{proposition}\label{semiweak}
Every semicompact sequence is weakly compact in $L^1(0,1)$.
\end{proposition}

\noindent
We also recall the well known Mazur's theorem (see e.g.\ \cite{ET}):

\begin{theorem}\label{mazur}
Let $E$ be a normed space, $(x_n)$ be a sequence in $E$ weakly converging to $x$. Then, there exists a sequence of convex linear combinations
\[y_n=\sum_{k=1}^n a_{n,k}x_k, \ a_{n,k}\in(0,1]\]
s.t.\ $y_n\to x$ (strongly) in $E$.
\end{theorem}

\noindent
We can now present our last example:

\begin{example}\label{zecca}
Let $\alpha,\beta:[0,1]\times\R\to\R$ be continuous functions s.t.\ $\alpha(t,s)\le\beta(t,s)$ for all $(t,s)\in[0,1]\times\R$, and for all $u\in\overline\Omega$ define $\phi(u)$ as the set of all functions $w\in F$ for which there exists $v\in L^1(0,1)$ s.t.\ for all $t\in[0,1]$
\[w(t)=\int_0^t v(\tau)\,d\tau,\]
and for a.e.\ $t\in[0,1]$
\[v(t)\in[\alpha(t,u(t)),\beta(t,u(t))].\]
Shortly, we might define $\phi(u)$ as a set-valued integral in the sense of Aumann
\[\phi(u)(t)=\int_0^t[\alpha(\tau,u(\tau)),\beta(\tau,u(\tau))]\,d\tau\]
(see \cite{BZ,KOZ}). Clearly, for all $u\in\overline\Omega$, any $w\in\phi(u)$ is absolutely continuous and hence a.e.\ differentiable in $[0,1]$ with derivative $v$. We first prove that $\phi$ has convex values. Let $u\in\overline\Omega$, $w_0,w_1\in\phi(u)$, and $\mu\in[0,1]$. There exist $v_0,v_1\in L^1(0,1)$ s.t.\
\[w_i(t)=\int_0^t v_i(\tau)\,d\tau, \ v_i(t)\in[\alpha(t,u(t)),\beta(t,u(t))] \ \text{(a.e.)}.\]
Set $w=(1-\mu)w_0+\mu w_1$ and $v=(1-\mu)v_0+\mu v_1$, then we have in $[0,1]$
\[w(t)=\int_0^t v(\tau)\,d\tau, \ v(t)\in[\alpha(t,u(t)),\beta(t,u(t))] \ \text{(a.e.)},\]
which implies $w\in\phi(u)$.
\vskip2pt
\noindent
We prove now that $\phi$ has compact values (this is not immediate and will require several steps). Let $u\in\overline\Omega$, $(w_n)$ be a sequence in $\phi(u)$, then there exists a sequence $(v_n)$ in $L^1(0,1)$ s.t.\ for all $n\in\N$
\[w_n(t)=\int_0^t v_n(\tau)\,d\tau, \ v_n(t)\in[\alpha(t,u(t)),\beta(t,u(t))] \ \text{(a.e.)}.\]
Clearly $(w_n)$ is bounded in $F$. Also, since $(v_n)$ is essentially bounded, $(w_n)$ turns out to be equi-absolutely continuous. By Ascoli's theorem, passing if necessary to a subsequence we have $w_n\to w$ in $F$ and $w$ is the primitive of some $v\in L^1(0,1)$, i.e., for all $t\in[0,1]$ we have
\beq\label{wv}
w(t)=\int_0^t v(\tau)\,d\tau.
\eeq
On the other side, $(v_n)$ is a semicompact sequence in $L^1(0,1)$ (Definition \ref{semicompact}), so by Proposition \ref{semiweak} we can pass to a further subsequence and have $(v_n)$ weakly converging in $L^1(0,1)$ to some $\hat v\in L^1(0,1)$. For all $t\in[0,1]$, the linear functional
\[f\mapsto\int_0^t f(\tau)\,d\tau\]
is bounded in $L^1(0,1)$, so weak convergence is enough to deduce that for all $t\in[0,1]$
\[w_n(t)=\int_0^t v_n(\tau)\,d\tau\to\int_0^t\hat v(\tau)\,d\tau.\]
Comparing to \eqref{wv}, we see that $v=\hat v$ in $L^1(0,1)$. So $(v_n)$ converges weakly to $v$ in $L^1(0,1)$. By Theorem \ref{mazur}, we can find a sequence $(v'_n)$ of convex linear combinations of $(v_n)$ s.t.\ $v'_n\to v$ in $L^1(0,1)$ (strongly). Clearly, for all $n\in\N$ and a.e.\ $t\in[0,1]$ we have
\[v'_n(t)\in[\alpha(t,u(t)),\beta(t,u(t))],\]
so we can pass to the limit and deduce the same property for $v$. So, by \eqref{wv}, we have $w\in\phi(u)$. Thus, $\phi(u)$ is compact. By similar arguments, we prove that $\phi$ is locally compact.
\vskip2pt
\noindent
Upper semicontinuity of $\phi$ can be proved as in the previous cases, applying \cite[Proposition 4.1.5]{G}. Nevertheless, in order to give the reader a more complete picture, we present a direct proof based on Definition \ref{usc}. Let $V\subset F$ be open, $\bar u\in\phi^+(V)$. We claim that there exists a neighborhood of $\bar u$ contained in $\phi^+(V)$. Indeed, by compactness of $\phi(\bar u)$, there exist $w_1,\ldots w_n\in\phi(\bar u)$ and $\eps_1,\ldots,\eps_n>0$ s.t.\
\begin{enumroman}
\item\label{zecca1} $B_{\eps_i}(w_i)\subseteq V$ ($i=1,\ldots n$);
\item\label{zecca2} $\phi(\bar u)\subset\cup_{i=1}^n B_{\eps_i/2}(w_i)$.
\end{enumroman}
Consider now the compact subset of $\R^2$
\[C=\big\{(t,y)\in\R^2:\,t\in[0,1],\,\alpha(t,\bar u(t))\le y\le\beta(t,\bar u(t))\big\},\]
and let $A$ be the open ball of $[0,1]\times\R$ centered at $C$ with radius $1$. Then set
\[\bar\eps:=\frac{\min\{\eps_1,\ldots\eps_n\}}{2}.\]
Since $\alpha$, $\beta$ are uniformly continuous in $\overline{A}$, we can find $\delta\in(0,1)$ s.t.\
\beq\label{unifcont}
\max\big\{|\alpha(t,y)-\alpha(t,z)|,\,|\beta(t,y)-\beta(t,z)|\big\}<\bar\eps \ \text{for all $(t,y),(t,z)\in\overline{A}$, ${\rm dist}((t,y),(t,z))<\delta$.}
\eeq
We claim that $\phi(B_\delta(\bar u))\subseteq V$. Indeed, fix $u\in B_\delta(\bar u)$. Clearly, we have $|u(t)-\bar u(t)|<\delta$ for all $t\in[0,1]$. Take now $w\in\phi(u)$, which can be written as
\[w(t)=\int_0^t v(\tau)\,d\tau,\]
with $v\in L^1(0,1)$ satisfying for a.e.\ $t\in[0,1]$
\[\alpha(t,u(t))\le v(t)\le\beta(t,u(t)).\]
If for a.e.\ $t\in[0,1]$
\[\alpha(t,\bar u(t))\le v(t)\le\beta(t,\bar u(t)),\]
then $w\in\phi(\bar u)\subseteq V$ and we are done. Otherwise, consider the truncated map $\bar v:[0,1]\to\R$ defined by
\[\bar v(t)=\begin{cases}
\alpha(t,\bar u(t)) & \text{if $v(t)<\alpha(t,\bar u(t))$} \\
v(t) & \text{if $\alpha(t,\bar u(t))\le v(t)\le\beta(t,\bar u(t))$} \\
\beta(t,\bar u(t)) & \text{if $v(t)>\beta(t,\bar u(t))$,}
\end{cases}\]
which is a $L^1$-function (see e.g.\ \cite{R}), and denote
\[\bar w(t)=\int_0^t\bar v(\tau)\,d\tau,\]
so $\bar w\in\phi(\bar u)$. By the bounds above we have for all $t\in[0,1]$ that $(t,u(t)),(t,\bar u(t))\in\overline{A}$ with
\[{\rm dist}((t,u(t)),(t,\bar u(t)))<\delta,\]
so by \eqref{unifcont} we have $|v(t)-\bar v(t)|<\bar\eps$ for a.e.\ $t\in[0,1]$. This in turn implies $\|w-\bar w\|_\infty<\bar\eps$. By \ref{zecca2}, we can find $i\in\{1,\ldots n\}$ s.t.\ $\bar w\in B_{\eps_i/2}(w_i)$. So, recalling the definition of $\bar\eps$, we have
\[\|w-w_i\|_\infty\le\|w-\bar w\|_\infty+\|\bar w-w_i\|_\infty <\eps_i,\]
hence by \ref{zecca1} $w\in V$. Thus, $\phi(B_\delta(\bar u))\subseteq V$ and $\phi$ turns out to be u.s.c. In conclusion, $\phi\in CJ(\overline\Omega,F)$ and it is locally compact.
\end{example}

\begin{remark}\label{pts}
Comparing Definition \ref{sol} and problem \eqref{odi}, one may be left in doubt that the non-trivial solutions ensured by Theorem \ref{bifodi} might be triples $(\pm 1,\eps,0)$ with $\eps\neq 0$ (quite trivial in fact). But this case may only occur if $0\in\phi(\pm1)$. Easy computations show that in Example \ref{gasinski} we have $0\notin\phi(\pm 1)$, due to the choice $\rho\in(0,1)$. Similarly, in Example \ref{benedetti} it is enough to choose functions $f$, $\alpha$, and $\beta$ to be positive in order to have $0\notin\phi(\pm 1)$, thus avoiding such difficulty. Also in Example \ref{zecca}, we can easily find $\alpha$, $\beta$ s.t.\ $0\notin\phi(\pm 1)$.
\end{remark}

\vskip4pt
\noindent
{\small {\bf Acknowledgement.} This paper has gone through a long process before being concluded, and several people deserve to be acknowledged for their contributions: we thank Leszek Gasi\'{n}ski for help on Example \ref{gasinski}, Irene Benedetti for Example \ref{benedetti}, Pietro Zecca and Valeri Obukhovskii for Example \ref{zecca}, and Sunra Mosconi for the figure. We also thank Oswaldo Rio Branco de Oliveira for fruitful discussions. A.\ Iannizzotto is member of GNAMPA (Gruppo Nazionale per l'Analisi Matematica, la Probabilit\`a e le loro Applicazioni) of INdAM (Istituto Nazionale di Alta Matematica 'Francesco Severi') and is partially supported by the research project {\em Integro-differential Equations and Non-Local Problems}, funded by Fondazione di Sardegna (2017).}

\end{document}